\documentclass[11pt, reqno]{amsart}
\usepackage{a4wide}
\usepackage{wrapfig}
\usepackage{amsmath}
\usepackage{amsthm}
\usepackage{mathrsfs}
\usepackage{color}
\usepackage{xcolor}
\usepackage{graphicx}
\usepackage{amssymb}
\usepackage{ dsfont }
\usepackage{url}
\usepackage{multicol}
\usepackage{tikz}
\usepackage{amsmath}
\usepackage{fancyhdr}
\usetikzlibrary{matrix}
\usetikzlibrary{arrows}
\usepackage{subfig}
\usepackage{hyperref}
\usepackage{listings}

\usepackage{esint}

\usepackage{enumitem}
\usepackage{varwidth}
\allowdisplaybreaks

\title[Orbital stability and instability for the $\phi^{4n}$-models]{Orbital stability and instability of periodic wave solutions for $\phi^{4n}$-models}

\author[G. Chen]{Gong Chen}
\address{Fields Institute for Research in Mathematical Sciences,  222 College St, Toronto, ON Canada, M5T 3J1}
\email{gchen@fields.utoronto.ca, gc@math.toronto.edu}
\thanks{G.C. was supported by Fields Institute for Research in Mathematical Sciences via Thematic Program on Mathematical Hydrodynamics
}

\author[J. M. Palacios]{Jos\'e M. Palacios}
\address{Institut Denis Poisson, Universit\'e de Tours, Universit\'e d'Orleans, CNRS, Parc Grandmont 37200, Tours, France}
\email{jose.palacios@lmpt.univ-tours.fr}

\newcommand{\be}{\begin{equation}}
\newcommand{\ee}{\end{equation}}
\newcommand{\bp}{\begin{proof}}
\newcommand{\ep}{\end{proof}}
\newcommand{\bel}{\begin{equation}\label}
\newcommand{\eeq}{\end{equation}}
\newcommand{\bea}{\begin{eqnarray}}
\newcommand{\eea}{\end{eqnarray}}
\newcommand{\bee}{\begin{eqnarray*}}
\newcommand{\eee}{\end{eqnarray*}}
\newcommand{\ben}{\begin{enumerate}}
\newcommand{\een}{\end{enumerate}}

\newcommand{\R}{\mathbb{R}}

\newcommand{\N}{\mathbb{N}}
\newcommand{\Z}{\mathbb{Z}}
\newcommand{\T}{\mathbb{T}}

\newtheorem{thm}{Theorem}[section]

\newtheorem{lem}[thm]{Lemma}
\newtheorem{prop}[thm]{Proposition}

\theoremstyle{remark}
\newtheorem{rem}{Remark}[section]

\definecolor{codegreen}{rgb}{0,0.6,0}
\definecolor{codegray}{rgb}{0.5,0.5,0.5}
\definecolor{codepurple}{rgb}{0.58,0,0.82}
\definecolor{backcolour}{rgb}{0.95,0.95,0.92}

\lstdefinestyle{mystyle}{
	backgroundcolor=\color{backcolour},   
	commentstyle=\color{codegreen},
	keywordstyle=\color{magenta},
	numberstyle=\tiny\color{codegray},
	stringstyle=\color{codepurple},
	basicstyle=\footnotesize,
	breakatwhitespace=false,         
	breaklines=true,                 
	captionpos=b,                    
	keepspaces=true,                 
	numbers=left,                    
	numbersep=5pt,                  
	showspaces=false,                
	showstringspaces=false,
	showtabs=false,                  
	tabsize=2
}
\numberwithin{equation}{section}

\usepackage{pgfplots}
\pgfplotsset{compat=newest}



\theoremstyle{definition}

\numberwithin{ej}{section}

\begin{document}
%\begin{sf}
%==============PORTADA=====================

%==============CAMBIO DE MARGENES=====================
%\newgeometry{top=2.5cm, bottom=2.5cm, right=2.0cm, left=1.8cm}

%==============ENCABEZA/PIE DE PAGINA=====================
%\newpage
%\pagestyle{fancy}
%\fancyhf{}

%Encabezado
%\fancyhead[L]{\rightmark}
%\fancyhead[L]{\small \rm \textit{\curso}} %Izquierda
%\fancyhead[R]{\small \rm \textit{}} %Derecha

%\fancyfoot[L]{\small \rm \textit{\titulo.}} %Izquierda
%\fancyfoot[R]{\small \rm \textbf{\thepage}} %Derecha
%\fancyfoot[C]{\thepage} %Centro

\renewcommand{\sectionmark}[1]{\markright{\thesection.\ #1}}
\renewcommand{\headrulewidth}{0.5pt}
\renewcommand{\footrulewidth}{0.5pt}
\begin{abstract}
In this work we study the orbital stability/instability in the energy space of a specific family of periodic wave solutions of the general $\phi^{4n}$-model for all $n\in\N$. This family of periodic solutions are orbiting around the origin in the corresponding phase portrait and, in the standing case, are related (in a proper sense) with the aperiodic Kink solution that connect the states $-\tfrac{1}{2}$ with $\tfrac{1}{2}$. In the traveling case, we prove the orbital instability in the whole energy space for all $n\in\N$, while in the standing case we prove that, under some additional parity assumptions, these solutions are orbitally stable for all $n\in\N$. Furthermore, as a by-product of our analysis, we are able to extend the main result in \cite{deNa} (given for a different family of equations) to traveling wave solutions in the whole space, for all $n\in\N$.
\end{abstract}
\maketitle 
\section{Introduction}

\subsection{The model} In this work we seek to extend the analysis carried out by the second author in \cite{Pa}. Specifically, this paper is concerned with the stability properties of traveling/standing wave solutions to the $1+1$ dimensional $\phi^{4n}$-equation on the torus (see for example \cite{Lo}): \begin{align}\label{phif}
\partial_{t}^2\phi-\partial_x^2\phi=-\lambda_nV'_n(\phi), \quad t\in\R,\  x\in\T_L:=\R/L\Z,
\end{align}
where $\lambda_n\in\R$ is a positive parameter and $V_n(\phi)$ is given by the following class of potentials: \begin{align}\label{potential}
V_{n}(\phi):=\prod_{k=1}^n\left(\phi^2-v^2\big(k-\tfrac{1}{2}\big)^2\right)^2, \quad v>0.
\end{align}
Here, $\phi(t,x)$ denotes a real-valued $L$-periodic function. This family of equations corresponds to a  generalization of the celebrated $\phi^4$-equation in Quantum Field Theory, which arises as a model for self-interactions of scalar fields (represented by $\phi$). In particular, in the case $n=1$, equation \eqref{phif} is one of the simplest examples where to apply Feynman diagram techniques to do perturbative analysis in quantum theory.  

\medskip

The $\phi^4$-model has been extensively studied from both, a mathematical and a physical point of view. Especially, this equation has been a ``workhorse'' of the Ginzburg-Landau (phenomenological) theory of superconductivity, taking $\phi$ as the order parameter of the theory, that is, the macroscopic wave function of the condensed phase \cite{KeCu}. In particular, the $\phi^4$-equation has been derived as a simple continuum model of lightly doped polyacetylene \cite{Ri}. We refer the interested reader
to \cite{MaPa,PeSc,Va} for some other physical motivations.

\medskip

On the other hand, equation \eqref{phif} belongs to a bigger family of equations called the $P(\phi)_2$-theory, which considers general polynomial self-interactions of scalar fields, where the potential is assumed to be of the form $V(\phi)=(P(\phi))^2$, where $P(\cdot)$ corresponds to some polynomial and the potential $V$ is asked to be even. The first examples of such theory are the famous $\phi^4$, $\phi^6$ and $\phi^8$ models (notice that $\phi^6$ does not belongs to our current framework 
\eqref{potential}). In this setting, the self-interaction intensity is quantified by $V(\phi)$, and clearly sets the dynamics of the field \cite{Lo}.

\medskip

One interesting feature of the $\phi^{4n}$-model (and generally of the $P(\phi)_2$-theory) is that, as $n$ goes to infinity, for a proper selection of parameters $\lambda_n$, equation \eqref{phif} is converging to the so-called sine-Gordon equation \[
\partial_t^2\phi-\partial_x^2\phi+\sin\phi=0.
\] 
Roughly speaking, in order to recover the sine-Gordon  as a limiting equation of \eqref{phif}, the parameter $\lambda_n$ has to be chosen so that, for $n\in\N$ sufficiently large ($v=1$), \[
\lambda_n^{-1}=\pi^2\prod_{k=1}^n\big(k-\tfrac{1}{2}\big)^2+\varepsilon(n),
\]
where $\varepsilon:\N\to\R$ is any function converging to zero sufficiently fast as $n$ goes to infinity. Additionally, notice that, as $n$ increases, one is adding more and more different minima to the potential $V_n$ in \eqref{potential} (see Figure \ref{fig1}). Correspondingly, more soliton sectors. As a result, these polynomial theories are in general more difficult to handle than the sine-Gordon theory, although for $n$ large, one would expect the soliton properties to approach those of sine-Gordon solitons \cite{Lo}.

\medskip

From a mathematical point of view, equation \eqref{phif} can also be understood as a particular case of the general family of nonlinear Klein-Gordon equations: \begin{align}\label{klein}
\partial_t^2\phi-\partial_x^2\phi=m\phi+f(\phi),
\end{align}
where $m\in\R$ and $f:\R\to\R$ denotes the nonlinearity. Many important nonlinear models can be recovered as particular cases of this latter equation, such as the whole $\phi^{4n}$-family \eqref{phif}, as well as the $\phi^{4n+2}$-family and the sine-Gordon equations (see \cite{Lo} for the explicit form of the $\phi^{4n+2}$-family). Interestingly, under rather general assumptions it is still possible to obtain some stability results for model \eqref{klein}. We refer the reader to \cite{De,KMM2} for a fairly general theory for small solutions to equation \eqref{klein} and to \cite{LiSo,St} for studies of the long time asymptotics for some generalizations of equation \eqref{phif} with variable coefficients.

\medskip

On the other hand, since \eqref{phif} corresponds to a wave-like equation, it can be rewritten in the standard form as a first order system for $\vec{\phi}=(\phi_1,\phi_2)$ as \begin{align}\label{phif_2}
\begin{cases}
\partial_t\phi_1=\phi_2,
\\ \partial_t\phi_2=\partial_x^2\phi_1-\lambda_nV_n'(\phi_1).
\end{cases}
\end{align}
Moreover, from the Hamiltonian structure of the equation it follows that, at least formally, the energy of system \eqref{phif_2} is conserved along the trajectory, that is,
\begin{align}\label{energy} 
\mathcal{E}(\vec{\phi}(t))&:=\dfrac{1}{2}\int_0^L \big(\phi_2^2+\phi_{1,x}^2+2\lambda_nV_n(\phi_1)\big)(t,x)dx=\mathcal{E}(\vec{\phi}_0).
\end{align}
Besides, the conservation of momentum shall also play a fundamental role for our current purposes, which is given by:
\begin{align}\label{momentum}
\mathcal{P}(\vec{\phi}(t)):=\int_0^L \phi_2(t,x)\phi_{1,x}(t,x)dx=\mathcal{P}(\vec{\phi}_0).
\end{align}
We point out that, from these two conservation laws it follows that $H^1(\T_L)\times L^2(\T_L)$ defines the natural energy space associated to system \eqref{phif_2}.

\medskip

Additionally, equation \eqref{phif} is known for satisfying several symmetries. Among the most important ones we have the invariance under space and time translations. It is worth to notice that, in the aperiodic setting there is an extra invariance, the so-called  Lorentz boost, that means, if $\vec{\phi}(t,x)$ is a solution to the equation, then so is \[
\vec{\varphi}(t,x):=\vec{\phi}\big(\gamma(t-\beta x),\gamma(x-\beta t)\big) \quad \hbox{where}\quad \gamma^{-1}:= \sqrt{1-\beta^2} \quad \hbox{and}\quad \beta\in(-1,1). 
\]
However, this transformation does not let the period fixed, and hence, strictly speaking, it is not an invariance of the equation in our current setting.

\medskip

Now, in order to motivate our work we recall that, for general nonlinear evolution equations, two of the most important objects in nonlinear dynamics are traveling and standing wave solutions, particularly in the context of dispersive PDEs due to the so-called \emph{soliton conjecture}. The existence and (if the case) the corresponding orbital stability of such type of solutions have become a fundamental issue in the area. In this regard, we prove the existence of at least one branch of traveling wave solutions to equation \eqref{phif} in the periodic setting, as well as one associated branch of standing wave solutions. Nonetheless, we remark that, up to the best of our knowledge, for $n>2$ these solutions have no explicit form, which has been an important problem in this work.

\medskip

One of the key points in our analysis is the use of classical results of Grillakis-Shatah-Strauss (see \cite{GSS}) which set a general framework to study the orbital stability/instability for both traveling and standing wave solutions. These general results are based on the spectral information of the linearized Hamiltonian around these specific solutions. Thereby, it is worthwhile to notice that, in the real-valued case, equation \eqref{phif_2} can be rewritten in the abstract Hamiltonian form as \[
\partial_t\vec\phi=\mathbf{J}\mathcal{E}'(\vec\phi) \quad \hbox{where} \quad \mathbf{J}:=\left(\begin{matrix}
0 & 1 \\ -1 & 0 
\end{matrix}\right),
\]
where $\mathcal{E}'$ denotes the Frechet derivative of the conserved energy functional $\mathcal{E}$ in \eqref{energy}.

\medskip

Regarding the orbital stability of explicit solutions to equations \eqref{phif} and \eqref{klein}, there exists a vast literature regarding the aperiodic case. We refer the reader to \cite{HPW} for a classical and rather general result about orbital stability of Kink solutions for Klein-Gordon equations, and to \cite{KMM,KMMH} for some interesting results regarding asymptotic stability of Kink solutions for general scalar-field equations (see also \cite{AlMuPa3} for a recent work in this direction in the case of sine-Gordon). We also refer to \cite{Cu} for an study of the asymptotic stability properties of this type of solutions in dimension $3$. Nevertheless, for the periodic setting, there are not that many well-known results. We refer the reader to \cite{AnNa2,NaCa,NaPa} for the treatment of periodic solutions for a specific type of Klein-Gordon equations. Specifically, the first two of these works considers the stability problem of periodic solutions with $-\phi+\vert \phi\vert^4\phi $ as right-hand side in \eqref{phif}, while the third one considers $+\vert\phi\vert^2\phi$ and $-\phi+\vert\phi\vert^2\phi$ as right-hand sides. We emphasize that none of the $\phi^{4n}$ equations (for no $n\in\N$) fit any of these settings. On the other hand, as mentioned before, for the case $n=1$, the orbital in/stability of traveling/standing wave solutions to equation \eqref{phif} was already treated in \cite{Pa}. Regarding the stability of periodic wavetrains, we refer the reader to \cite{JoMaMiPl}. We remark that this latter result seems to be the first one (up to the best of our knowledge) for wavetrains in the periodic case (see also \cite{JoMaMiPl3}). On the other hand, we refer to \cite{DeMc,JoMaMiPl2} for stability results in a particularly interesting Klein-Gordon setting (but different from the previous-ones), the sine-Gordon equation. However, in the last two works, the authors are mostly focused in spectral and exponential stability, rather than in orbital stability. We point out that, in the previous case, the authors also deal with \emph{superluminal waves}, a case which we do not treat in this work. About the stability of periodic traveling waves in Hamiltonian equations that are first-order in time, we refer to \cite{DeUp} for stability results for the nonlinear Schr\"odinger equation and to \cite{An,DeKa,DeNi} for the KdV and mKdV settings. Finally, we refer the reader to \cite{AlMuPa} for an stability study for more complex periodic structures that do not fit into the framework of Grillakis \emph{et al.} \cite{GSS,GSS2}, such as spatiallty-periodic \emph{Breathers}. These are explicit solutions to the equation which behave as solitons but are also time-periodic. See also \cite{AlMuPa2,MP} for some stability results of aperiodic Breathers in the sine-Gordon equation.

\medskip

Finally, concerning the well-posedness of the equation, we recall that by applying the classical Kato theory for quasilinear equations we obtain the local well-posedness in the energy space $H^1(\T_L)\times L^2(\T_L)$ of equation \eqref{phif} (see \cite{Kato}).  We refer the reader to  \cite{De,HaNa,Kl,Kl2} for several other local and global well-posedness results in one-dimensional and higher dimensional Klein-Gordon equations.

\begin{figure}[h!]
   \centering
   \includegraphics[scale=0.445]{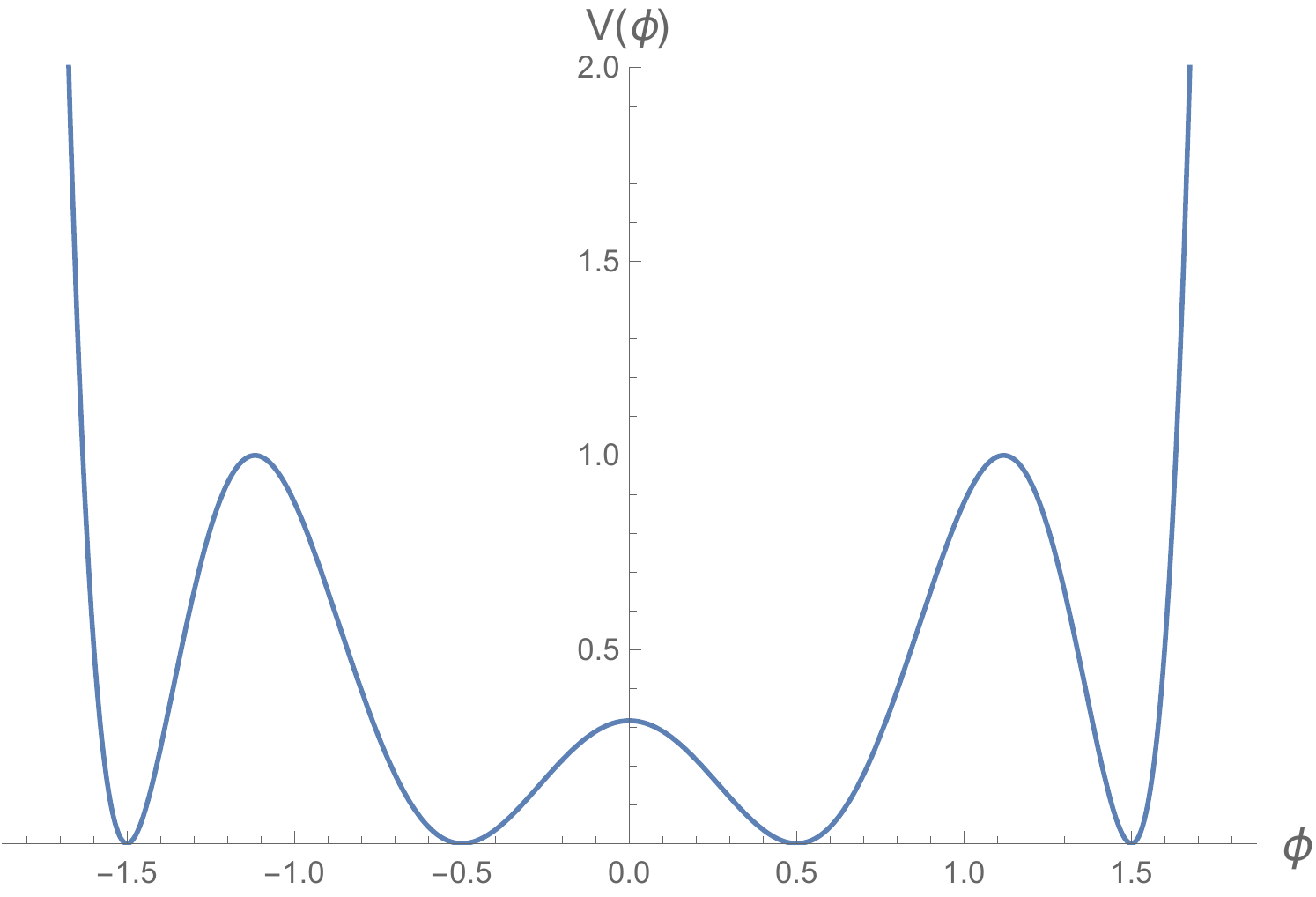}\qquad \quad 
   \includegraphics[scale=0.445]{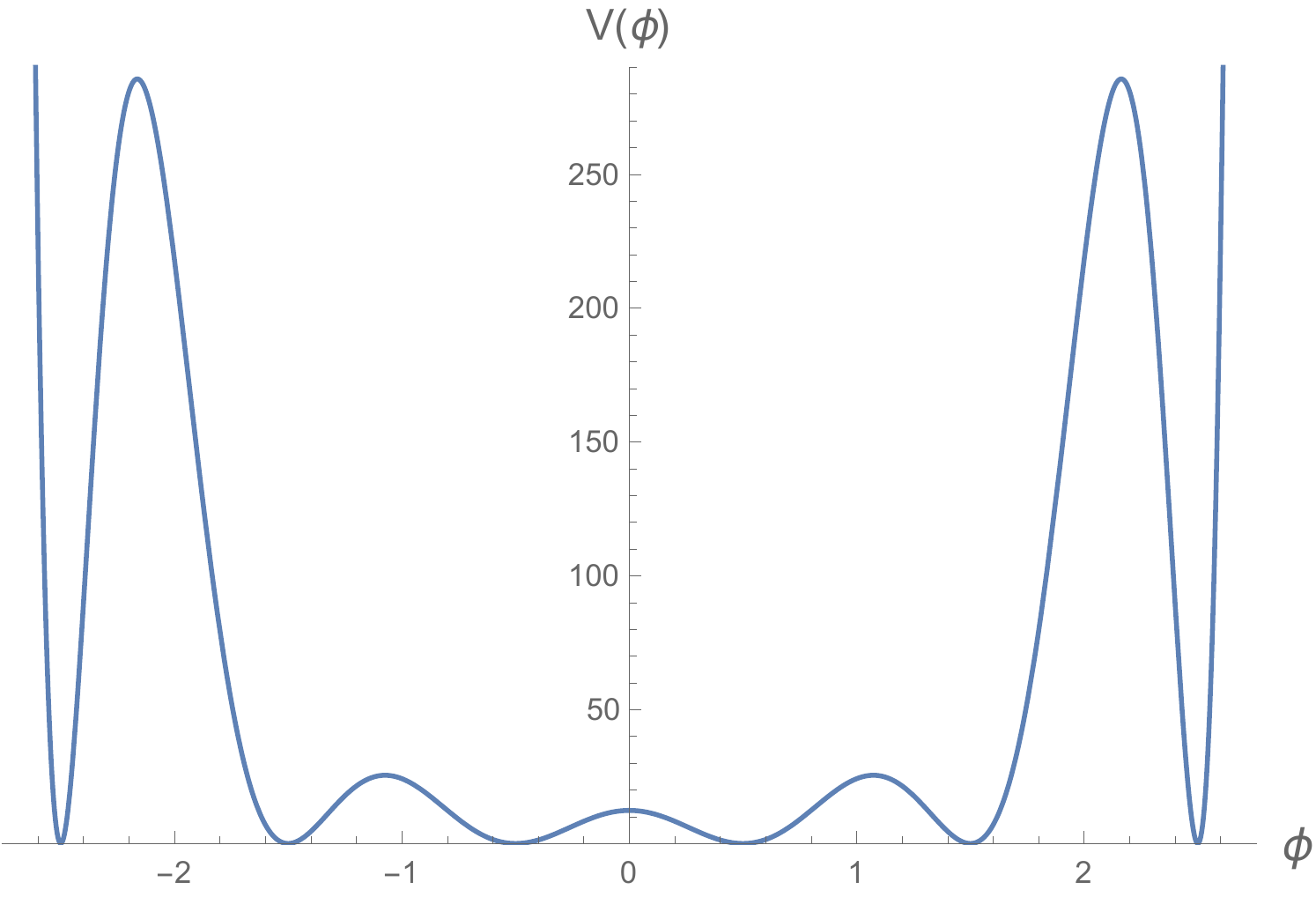}%%{DtWKmenosQ.pdf}  % requires the graphicx package
   \caption{On the left-hand we have $V_{n}(\phi)$ for $n=2$, that is, the potential associated to the $\phi^8$-model. On the right-hand we have $V_{n}$ for $n=3$, that is, the potential associated to the $\phi^{12}$-model.}
   \label{fig1}
\end{figure}

\subsection{Main results}

In order to present our main results, let us first define what it means for a solution to be \emph{Orbitally Stable}. We say that a traveling wave solution $\vec{\varphi}_c$ is orbitally stable if for all $\varepsilon>0$ there exists $\delta>0$ small enough such that for every initial data $\vec{\phi}_0\in X$, with $X:=H^1(\T_L)\times L^2(\T_L)$, satisfying $
\Vert \vec{\phi}_0-\vec{\varphi}_c\Vert_{X}\leq\delta$, then
\[
\sup_{t\in\R}\inf_{\rho\in[0,L)}\Vert\vec{\phi}(t)-\vec{\varphi}_c(\cdot-\rho) \Vert_{X}<\varepsilon.
\]
Additionally, we shall say that an odd-standing wave solution $\vec{\varphi}$ is orbitally stable in the \emph{odd energy space} $X_{\mathrm{odd}}:=H^1_{\mathrm{odd}}(\T_L)\times L^2_{\mathrm{odd}}(\T_L)$ if, for all $\varepsilon>0$, there exists $\delta>0$ small enough such that for every initial data $\vec{\phi}_0\in X_{\mathrm{odd}}$ satisfying $
\Vert \vec{\phi}_0-\vec{\varphi}\Vert_{X}\leq\delta$, then
\[
\sup_{t\in\R}\Vert\vec{\phi}(t)-\vec{\varphi} \Vert_{X}<\varepsilon.
\] 
Otherwise, we say that $\vec{\varphi}_c$ (respectively $\vec{\varphi}$) is orbitally unstable. In particular, this latter is the case when the solution ceases to exist in finite time.

\medskip

It is worth noticing that, even when it is not explicitly said, we shall always assume that $L$ is the fundamental period of $\vec{\varphi}_c$. In particular, we are only considering perturbations with exactly the same period as our fundamental solution.

\medskip

Now, in order to avoid overly introducing new notation and definitions in this introductory section, we shall only formally state our main results. We remark again that all the theorems below have already been proven in \cite{Pa} for the case $n=1$. Moreover, since there is no explicit solution for $n>2$, in the sequel, we shall refer to the specific family of solutions we are considering as ``periodic solutions orbiting around the origin'' (see section \ref{existence} below for further details).

\begin{thm}[Orbital instability of subluminal traveling waves]\label{MT1}
Let $n\in\N$ be arbitrary but fixed. Then, periodic traveling wave solutions ($c\in(-1,1)$) orbiting around the origin in the corresponding phase-portrait are orbitally unstable in the energy space by the periodic flow of the $\phi^{4n}$ equation.
\end{thm}

\begin{rem}
We refer the reader to Figure \ref{fig2} for a quick qualitative checking of the behavior of solutions to model \eqref{phif} around the origin in the corresponding phase-portrait.
\end{rem}

As discussed above, in order to obtain this result we use the general theory of Grillakis-Shatah-Strauss. Nevertheless, the results in \cite{GSS} require the existence of a non-trivial curve of solutions of the form $c\mapsto \phi_c$, which, in sharp contrast with the aperiodic setting, presents a delicate issue to overcome, and most of this work is devoted to address this problem.

\begin{thm}[Existence of a smooth curve of solutions]
Consider $n\in\N$ and let $L>0$ be arbitrary but fixed. There exists a non-trivial smooth curve of periodic solutions $c\mapsto \phi_c\in H^\infty(\T_L)$ orbiting around the origin in the corresponding phase-portrait.
\end{thm}

\begin{rem}
We point out that the domain on which $c$ is moving in the definition of $c\mapsto \phi_c$ is not always equals to $(-1,1)$ (see Theorem \ref{thm_monotonicity} below for further details). 
\end{rem}

The main obstruction in showing the previous theorem is due to both, the difficulty to handle the potential $V_n(\phi)$ for general $n\in\N$, as well as the fact that, for $n>2$, no explicit solution exists. In order to surpass this problem we use ODE results for Hamiltonian systems and several  combinatorial arguments to handle the potential.

\medskip

Notice that from the orbital instability theorem above we also conclude that the associated stationary solutions ($c=0$) are orbitally unstable. However, under some additional hypothesis we have the following result.

\begin{thm}[Orbital stability: stationary case]\label{MT2}
Let $n\in\N$ be arbitrary but fixed. Then, periodic standing wave solution ($c=0$) orbiting around the origin in the corresponding phase-portrait are orbitally stable by the periodic flow of the $\phi^{4n}$ equation under $(\mathrm{odd},\mathrm{odd})$ perturbations in the energy space.
\end{thm}

Finally, as a by-product of our analysis we are able to extend the main result in \cite{deNa} (given only for cases $n=1,2$, see Section \ref{ext_nata} below for more details), for equation \eqref{phi_w} below, to all $n\in\N$.
\begin{thm}[Orbital instability of traveling waves in \cite{deNa}]\label{MT3}
Let $n\in\N$ be arbitrary but fixed. Then, traveling wave solutions ($c\in(-1,1)$) found in \cite{deNa} orbiting around the origin in the corresponding phase-portrait associated to equation \eqref{phi_w} are orbitally unstable in the energy space.
\end{thm}

\begin{rem}
We emphasize that the previous theorems are independent of the results in \cite{Pa} and have been proven by different techniques.
\end{rem}

\begin{rem}
As an important observation we point out that Theorem \ref{MT2} is motivated by the fact that the oddness character of the initial data is preserved by the periodic flow associated to equation \eqref{phif}. In other words, if $\vec{\phi}_0=(\phi_{0,1},\phi_{0,2})=(\mathrm{odd},\mathrm{odd})$, then so is the solution for all times. Then, noticing that, under the additional requirement $\phi(x=0)=0$, the solution orbiting around zero in the corresponding phase-portrait correspond to an odd function. Thus, in the case $c=0$, the associated solution corresponds to an $(\mathrm{odd},\mathrm{odd})$ vector, and hence, under the assumptions of the previous theorem, the solution associated to this kind of initial perturbation shall always remain odd.  Here, and for the rest of this paper, when we refer to an \emph{odd} function, we mean that it is odd regarded as a function in the whole line.
\end{rem}

\begin{rem}
We point out that, since equation \eqref{phif} (equation \eqref{phi_w} for Theorem \ref{MT3}) is also invariant under the maps: \[
u(t,x)\mapsto u(-t,x),\quad u(t,x)\mapsto -u(t,x) \quad \hbox{and} \quad u(t,x)\mapsto -u(-t,x),
\]
we also deduce Theorems \ref{MT1}, \ref{MT2} and \ref{MT3} for both traveling and \emph{anti-traveling}\footnote{The solution with a minus sign in front (which is also a solution).} wave solutions, moving to the left or right respectively.
\end{rem}

\subsection{Organization of this paper} 

This paper is organized as follow. In Section \ref{existence} we prove the existence of a smooth curve of traveling waves solutions, and show that, under some conditions on the size of the period, we are able to consider standing waves solutions too. In Section \ref{sec:SpecAna} we provide the main spectral information of the linear operators needed in the stability analysis. Then, in Section \ref{stab_standing} we use the spectral information to conclude stability of standing waves under odd perturbations. In Section \ref{sec:Instability} we show the instability of traveling waves in the whole energy space. Finally, in Section \ref{ext_nata} we extend the analysis in \cite{deNa} to traveling waves solutions.

\medskip

\section{Existence of smooth curves periodic solutions}\label{existence}

In this section we seek to establish the existence of smooth curves of periodic traveling wave solutions to equation \eqref{phif} associated to subluminal waves, that is, with speed $c\in(-1,1)$. More precisely, in this section we look for solutions of the form $\phi(t,x)=\phi_c(x-ct)$. Before going further notice that, with no loss of generality, from now on we can assume\footnote{If not, we use the transformation $(t,x)\mapsto(\lambda_n^{1/2}t,\lambda_n^{1/2}x)$ what fixes $\lambda_n=1$. To fix $v^2=1$ it is enough to re-scale $\phi$ by defining the change of variables $\varphi(t,x)=v\phi\big(v^{2n-1}t,v^{2n-1}x\big)$.} $\lambda_n=v^2=1$. Thus, plugging $\phi_c(x-ct)$ into the equation, we obtain that if $\phi(t,x)$ is a traveling wave solution, then $\phi_c$ must satisfy: \begin{align}\label{solit_eq}
(c^2-1)\phi_c''=-V'_{n}(\phi_c).
\end{align}
On the other hand, the question regarding the existence of periodic solutions for the latter equation can be rewritten in terms of the following (autonomous) Hamiltonian system:
\begin{align}\label{system}
\begin{cases}\dot{u}=v,
\\ \dot{v}=\tfrac{1}{\omega}V_{n}'(u),
\end{cases}
\end{align}
where $\omega:=1-c^2$. From the explicit form of $V_{n}$ in \eqref{potential} it follows that the previous system has exactly $4n-1$ critical points. In fact, first of all, since $V_n'$ is a $(4n-1)$-th degree polynomial (see \eqref{comp_v_x}), it follows that it can have at most $4n-1$ real roots. Now, from direct computations, recalling the explicit form of $V_n$ in \eqref{potential}, we infer that zero is a simple real root\footnote{From the explicit form of $V_n$ it immediately follows that $V_n'$ has a factor $x$ multiplying the whole expression.} of $V_n'$. Besides, it is not hard to see that each root associated to each individual factor in the definition of $V_n$ is also a simple root\footnote{Since each individual factor in $V_n$ is of the form $(x^2-a^2)^2$, its derivative still contains a factor $(x^2-a^2)$. Therefore, $x=\pm a$ is still a root of $V_n'$.} of $V_n'$. Summarizing, we have found $2n+1$ roots of $V_n'$, which are precisely located at \begin{align}\label{critical_points}
(u_{-k},v_{-k}):=\big(-k+\tfrac{1}{2},0\big), \quad  (u_0,v_0):=(0,0), \quad  (u_k,v_k):=\big(k-\tfrac{1}{2},0\big),
\end{align}
where $k=1,...,n$. Even more,   the remaining $2n-2$ critical points are located in between each consecutive pair\footnote{This follows, for example, from Rolle Theorem applied to $f(x)=V_n(x)$, since all of these critical points in \eqref{critical_points} (except for $x=0$) are also roots of $V_n$. Thus, $f'(x)$ must to have at least one root in between each pair.} in \eqref{critical_points} for $\pm k=1,...,n$. More specifically, for each $k\in\{1,..,n\}$, we have exactly one critical point in between $(u_k,v_k)$ and $(u_{k+1},v_{k+1})$ (and their corresponding reflections, that is, in between each pair $(u_{-k-1},v_{-k-1})$ and $(u_{-k},v_{-k})$). Since we already have found $4n-1$ roots, there cannot be any other missing root for $V_n'$. Hence, the two nearest critical points to $(0,0)$ are $(u_{\pm1},v_{\pm 1})$ given in \eqref{critical_points}.  Moreover, by standard computations we see that the linearized matrix around each of these points takes the form \begin{align}\label{form_linear_matrix}
M:=\dfrac{1}{\omega}\left(\begin{matrix}
0 & \omega 
\\ V_{n}'' & 0
\end{matrix}\right).
\end{align}
Furthermore, from direct computations it follows that, for all $n\in\N$ we have (see \eqref{vnpp} below): \[
V_n''(0)=-4\prod_{k=1}^n\big(k-\tfrac{1}{2}\big)^4\sum_{k=1}^n\big(k-\tfrac{1}{2}\big)^{-2}<0.
\]
Thus, for $c\in(-1,1)$ or equivalently for $\omega>0$, from the latter inequality, and recalling identity  \eqref{form_linear_matrix}, it follows that $(0,0)$ is a stable center point for all $n\in\N$. Even more, from similar computations it is not hard to see that $V_{n}''\big(\tfrac{1}{2}\big)>0$, and hence, $(u_{\pm1},v_{\pm1})$ are both saddle critical points, for all $n\in\N$. 

\medskip

On the other hand, recalling that the previous system is Hamiltonian, and setting $(0,0)$ as the zero energy level, we obtain that the Hamiltonian associated to \eqref{system} is given by \begin{align}\label{hamilt}
\mathcal{H}(u,v):=\dfrac{1}{2}v^2-\dfrac{1}{\omega}\big(V_{n}(u)-V_{n}(0)\big).
\end{align}
Therefore, by the standard ODE theory for Hamiltonian equations (see for example \cite{ChiL}), we know that all periodic solutions of \eqref{system} orbiting around $(0,0)$ corresponds to regular level sets of $\mathcal{H}$ given by \begin{align}\label{gamma}
\Gamma_\beta:=\big\{(u,v): \ \mathcal{H}(u,v)=\beta\big\},
\end{align}
with $\beta\in(0,E_\star)$, where the maximal energy level $E_\star$ is given by \[
E_\star:=-\dfrac{1}{\omega}\big(V(\tfrac{1}{2})-V(0)\big)=\dfrac{1}{\omega}\prod_{k=1}^n\big(k-\tfrac{1}{2}\big)^4.
\]
\begin{figure}[h!]
   \centering
   \includegraphics[scale=0.26]{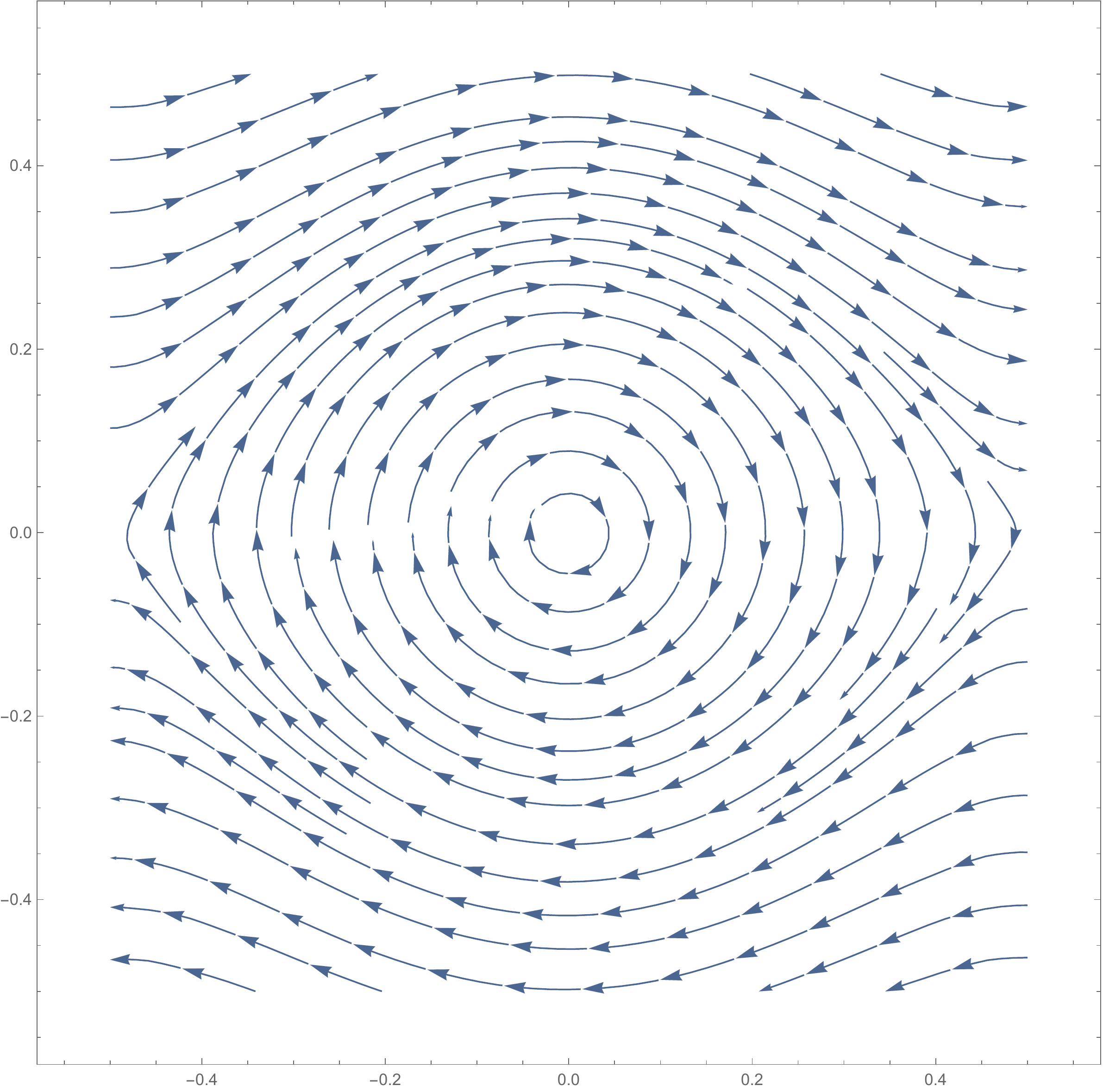} \qquad
   \includegraphics[scale=0.26]{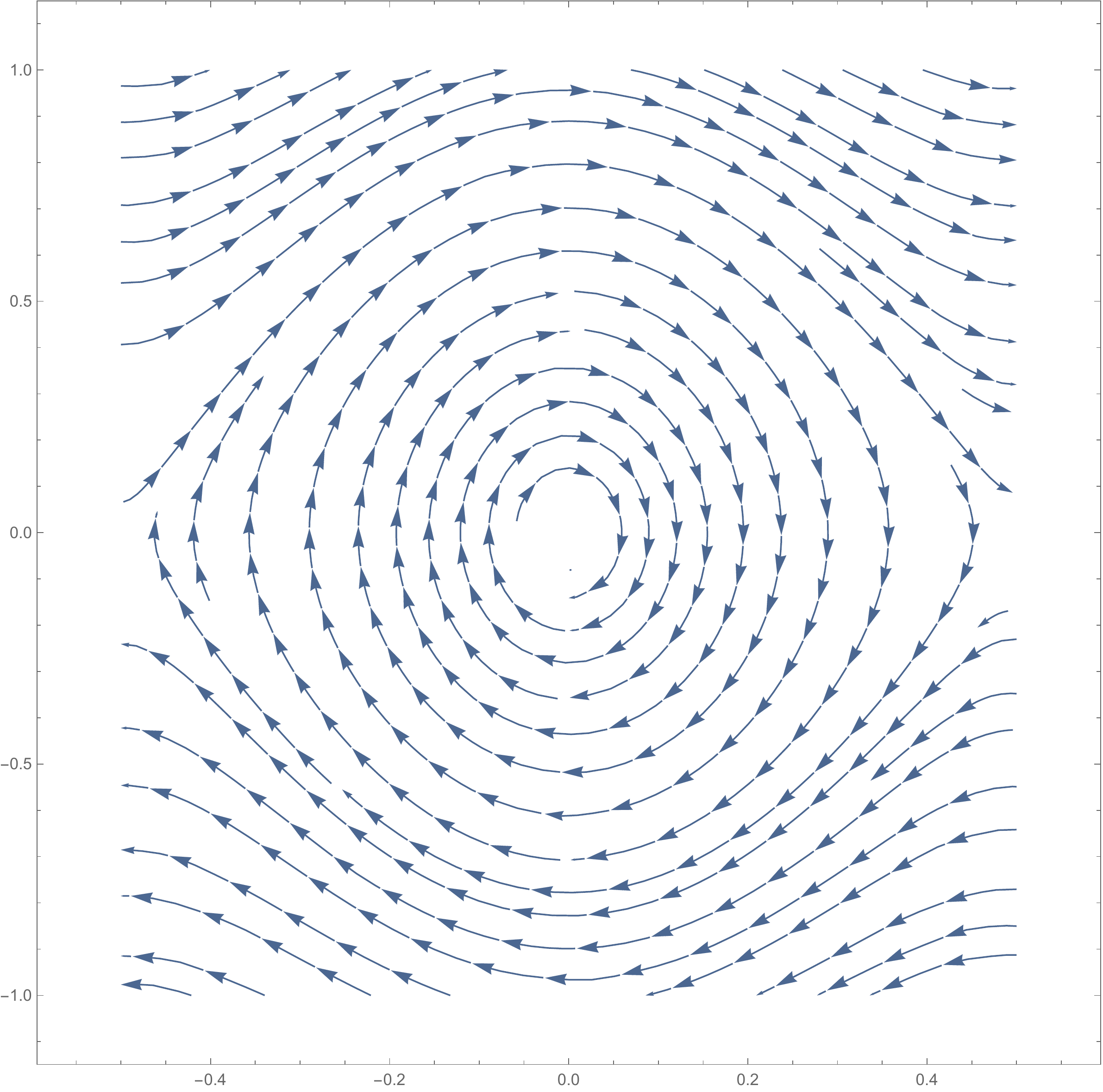}%%{DtWKmenosQ.pdf}  % requires the graphicx package
   \caption{Phase portrait of the Hamiltonian system \eqref{system} around $(0,0)$ for the first two cases $n=1,2$. On the left we have the phase portrait associated to the $\phi^4$-model while on the right the one associated to $\phi^8$.}\label{fig2}
\end{figure}
Now, with the additional constraint $\phi_c(0)=0$, from the symmetry of these level sets it follows that all solutions associated to these periodic orbits are $\mathrm{odd}$ (other solutions are translations of the same function, and consequently, not necessarily odd). Finally, by using again that each solution is a level curve of $\mathcal{H}$ and the symmetry of the phase portrait, it follows from \eqref{hamilt}-\eqref{gamma} that, for every $\beta\in(0,E_\star)$, the period of the corresponding odd solution satisfies
\begin{equation}
L=\sqrt{2}\int_{x_0}^{x_1}\dfrac{dx}{\sqrt{\beta+\frac{1}{\omega}(V_n(x)-V_n(0))}},\label{eq:peri}
\end{equation}
where $x_0$ and $x_1$ are the left and right intersections of the curve given by $\tfrac{1}{2}v^2-\tfrac{1}{\omega}\big(V_n(u)-V(0)\big)=\beta$ with the $u$-axis. We point out that the upper integration limit $x_1$ can also be written as the solution of $V(x)=V(0)-\omega\beta$ for $x\in(0,\tfrac{1}{2})$, and $x_0=-x_1$ (note that there is only one solution in this interval). Moreover, from the equation for $x_1$ we also infer that when $\beta$ goes to zero (or $c$ goes to $1$ for fixed $\beta$), $x_1$ goes to zero too. It is worth noting that the period $L$ defines a convergent improper integral for all values of $\beta\in(0,E_\star)$. Furthermore, notice that \[
\lim_{\beta\to E_\star^-}L(\beta)=+\infty.
\]
On the other hand, when $\beta\to 0^+$ we have\footnote{If the reader prefers, the existence of this limit can be rigorously justify by defining it (the limit) after the proof of the monotonicity of the period. Notice that the period is trivially bounded from below by $0$ and decreases as $\beta\to0^+$ (see the proof of Theorem \ref{thm_monotonicity} below). Hence, $L(\beta)$ has a limit as $\beta\to0^+$.}: \[
\lim_{\beta\to0^+}L(\beta)=\sqrt{2\omega}\lim_{\beta\to0^+}\int_{x_0(\beta)}^{x_1(\beta)}\dfrac{dx}{\sqrt{\omega\beta+V_n(x)-V_n(0)}}=:\sqrt{\omega}\delta_n,
\]
where $\delta_n\in[0,\infty)$ does not depends on $\omega$. The following theorem ensures us that, once we fix the period $L\in(0,\infty)$, the previous method produces a non-trivial smooth curve of periodic traveling wave solutions that can be parameterized by their speeds.

\begin{thm}[Smooth curve of periodic solutions]\label{thm_monotonicity}
Consider $n\in\N$ and let $L>0$ be arbitrary but fixed. Then, for any speed $c$ satisfying \[
c\in(-1,1) \quad \hbox{such that}\quad L>\sqrt{\omega}\delta_n,
\] 
there exists an unique energy level $\beta=\beta(c)\in(0,E_\star)$ such that the periodic wave solution $\vec{\phi}(x-ct)$ to the $\phi^{4n}$-equation \eqref{phif} constructed above has fundamental period $L$. Furthermore, the map $c\mapsto \phi_c(t=0,x)\in H^1(\T_L)$ is smooth.
\end{thm}

\begin{rem}
Notice that, by choosing $L>\delta_n$ we are able to consider standing waves solutions. These standing waves are related (in some sense) to the odd Kink solution of the $\phi^{4n}$-model. Additionally, when $c=0$, the corresponding solution is $(\phi_1,\phi_2)=(\mathrm{odd},\mathrm{odd})$, while when $c\neq 0$, the solution is $(\phi_1,\phi_2)=(\mathrm{odd},\mathrm{even})$, property that is not preserved by the flow.
\end{rem}

The main theorem in \cite{Chi} ensures that, under our current notations, if $-(V_n(x)-V_n(0))/(V_n'(x))^2$ is strictly convex for $x\in(-\tfrac{1}{2},\tfrac{1}{2})$, then the period $L=L(\beta)$ defines a strictly increasing function of $\beta$. Besides, notice that, by showing the strict monotonicity of $L$ with respect to the energy level $\beta$, the proof of the theorem follows. Thus, in order to conclude the proof of the theorem, it is enough to study the sign of the following function:
\[
-\dfrac{d^2}{dx^2}\dfrac{V_{n}(x)-V_n(0)}{(V_{n}'(x))^2}=\dfrac{3V_{n}''\big( (V_{n}')^2-2(V_{n}-V_n(0))V_{n}''\big)+2(V_{n}-V_n(0))V_{n}'V_{n}'''}{(V_{n}')^4}.
\]
Then, our first goal is to show the non-negativity of the latter quantity. Since the denominator is always non-negative, for this first step it is enough to show that \begin{align*}%\label{red_convexity}
\mathcal{V}_n(x):=3V_{n}''(x)\big( (V_{n}'(x))^2-2(V_{n}(x)-V_n(0))V_{n}''(x)\big)+2(V_{n}(x)-V_n(0))V_{n}'(x)V_{n}'''(x)\geq 0.
\end{align*}
In order to show that the latter inequality holds, we start by doing several basic computations needed in our analysis. First of all, by directly differentiating $V_n$ we have
\begin{align}\label{comp_v_x}
V_{n}'=4x\prod_{k=1}^n\left(x^2-\big(k-\tfrac{1}{2}\big)^2\right)\sum_{P\in\mathcal{P}_{n-1}^n}\prod_{i\in P}\left(x^2-\big(i-\tfrac{1}{2}\big)^2\right),
\end{align}
where $\mathcal{P}_{m}^n$ denotes the set\footnote{We call $m$-combination of a set $E$ to any subset of $m$ different elements $E$. For example, \[\mathcal{P}_2^4=\big\{\{1,2\},\{1,3\},\{1,4\},\{2,3\},\{2,4\},\{3,4\}\big\}.\]} of $m$-combinations of $\{1,...,n\}$ without repetitions and no permutations allowed. In particular, each $P\in\mathcal{P}_{n-1}^n$ is a set of $(n-1)$ elements. For the sake of clarity, let us introduce some notation that shall be useful in the sequel.  From now on we shall denote by $\Pi_{n}$, $\Pi_{n,0}$ and $\Sigma_{i}$ the following quantities\footnote{By convention $\Sigma_0=1$ and $\Sigma_m=0$ for $m<0$.}
\[
\Pi_n:=\prod_{k=1}^n\left(x^2-\big(k-\tfrac{1}{2}\big)^2\right), \ \ \Pi_{n,0}:=\prod_{k=1}^n\big(k-\tfrac{1}{2}\big)^2 \ \ \hbox{and}\ \  \Sigma_i:=\sum_{P\in\mathcal{P}_{i}^n}\prod_{j\in P}\left(x^2-\big(j-\tfrac{1}{2}\big)^2\right).
\]
Hence, by taking advantages of the previous notations we can write, for example, $
V_n'=4x\Pi_n\Sigma_{n-1}$. Then, performing similar direct computations and taking advantage of the previous notations, we are able to express $V_n''$ and $V_n'''$ as:
\begin{align}
V_n''&=4\Pi_n\Sigma_{n-1}+8x^2\Sigma_{n-1}^2+16x^2\Pi_n\Sigma_{n-2}, \label{vnpp}
\\ V_n'''&=24x\Sigma_{n-1}^2+48x\Pi_n\Sigma_{n-2}+96x^3\Sigma_{n-1}\Sigma_{n-2}+96x^3\Pi_n\Sigma_{n-3}.\nonumber
\end{align}
Therefore, gathering the identities above and performing some extra direct computations we obtain $\tfrac{1}{96}\mathcal{V}_n =\mathbf{A}+\mathbf{B}x^2+\mathbf{C}x^4$, where \begin{align*}
\mathbf{A}&:=-\big(\Pi_n^2-\Pi_{n,0}^2\big)\Pi_n^2\Sigma_{n-1}^2,
\\ \mathbf{B}&:=2\Pi_{n,0}^2\Pi_n\Sigma_{n-1}^3-4(\Pi_n^2-\Pi_{n,0}^2)\Pi_n^2\Sigma_{n-1}\Sigma_{n-2},
\\ \mathbf{C}&:=4\Pi_{n,0}^2\Sigma_{n-1}^4
+8\Pi_{n,0}^2\Pi_n\Sigma_{n-1}^2\Sigma_{n-2}+8(\Pi_n^2-\Pi_{n,0}^2)\Pi_n^2\Sigma_{n-1}\Sigma_{n-3}
\\ & \quad \ \, -16(\Pi_n^2-\Pi_{n,0}^2)\Pi_n^2\Sigma_{n-2}^2.
\end{align*}
Now, for the sake of clarity we split the analysis into several  small lemmas. Moreover, since the case $n=1$ was already treated in \cite{Pa}, from now on we shall only address the case $n>1$. The following lemma give us the non-negativity of the sum of the second term in $\mathbf{B}$ with the second one in $\mathbf{C}$ (notice that the terms associated to $\mathbf{C}$ in $\mathcal{V}_n$ have an extra $x^2$ with respect to the ones associated to $\mathbf{B}$). 
\begin{lem}\label{mon_lem_1} Let $n\in\N$ with $n\geq 2$. Then, for all $x\in(-\tfrac{1}{2},\tfrac{1}{2})$ we have: \begin{align}
\label{first_prop_ineq}
-4\big(\Pi_n^2-\Pi_{n,0}^2\big)\Pi_n^2\Sigma_{n-1}\Sigma_{n-2}+8x^2\Pi_{n,0}^2\Pi_n\Sigma_{n-1}^2\Sigma_{n-2}\geq0.
\end{align}
\end{lem}

\begin{proof} In fact, first of all, in order to simplify the proof we start by factorizing the left-hand side of inequality \eqref{first_prop_ineq} by $4\Pi_n\Sigma_{n-1}\Sigma_{n-2}$. Although, notice that, for all $n\in\N$, if we expand all terms involved in $4\Pi_n\Sigma_{n-1}\Sigma_{n-2}$, by using the definition of $\Pi_n$, $\Sigma_{n-1}$ and $\Sigma_{n-2}$ it is not difficult to see  that each addend in the resulting multiplication $4\Pi_n\Sigma_{n-1}\Sigma_{n-2}$ is composed by exactly $3n-3$ factors\footnote{In fact, notice that each addend in $\Sigma_{i}$ is composed exactly by $i$ factors, and that $\Pi_n$ is composed by $n$ more factors. Hence, each addend in the composition $4\Pi_n\Sigma_{n-1}\Sigma_{n-2}$ has exactly $3n-3$ factors.}, each of which is simultaneously negative for all $x\in(-\tfrac{1}{2},\tfrac{1}{2})$. This latter remark comes from the fact that, for all $k\geq 1$, any factor of the form $x^2-(k-\tfrac{1}{2})^2$ is non-positive for all $x\in(-\tfrac{1}{2},\tfrac{1}{2})$. Thus, inequality \eqref{first_prop_ineq} is equivalent to show that, for any $n\in\N$ with $n\geq 2$, and all $x\in(-\tfrac{1}{2},\tfrac{1}{2})$ the following holds:  \begin{align}\label{reduction_first_prop}
-\big(\Pi_n^2-\Pi_{n,0}^2\big)\Pi_n+2x^2\Pi_{n,0}^2 \Sigma_{n-1}\gtreqless0,
\end{align}
where we have to choose the ``$\leq $'' sign in the latter inequality whenever $n$ is even, and the ``$\geq $'' sign otherwise. Of course, the change from ``$\leq$'' to ``$\geq$'' comes from the fact that $3(n-1)$ is even whenever $n$ is odd, and odd whenever $n$ is even. Consequently, if $n$ is even, the function $\Pi_n\Sigma_{n-1}\Sigma_{n-2}$ is non-positive for all $x\in(-\tfrac{1}{2},\tfrac{1}{2})$, while it is non-negative if $n$ is odd.

\medskip

\textbf{Case $n$ even:} In this case we are lead to prove inequality \eqref{reduction_first_prop} with ``$\leq$''-sign. In fact, let us start by defining \[
f(x):=\big(\Pi_n^2-\Pi_{n,0}^2\big)\Pi_n-2x^2\Pi_{n,0}^2\Sigma_{n-1}.
\]
By definition it immediately follows that $f(0)=0$ and that $f(x)$ is an even function. Thus, it is enough to show that \begin{align}\label{fprop_prime}
f'(x)=6x\big(\Pi_n^2-\Pi_{n,0}^2\big)\Sigma_{n-1}-8x^3\Pi_{n,0}^2\Sigma_{n-2}\geq 0, \quad \,\hbox{for all }\, x\in(0,\tfrac{1}{2}).
\end{align}
Now, in order to prove the latter inequality, it is enough to recall the following basic property: If $a,b,c,d\in\R$ are all positive numbers satisfying $$
a\geq c \quad \hbox{and}\quad b\geq d,
$$
then $ab\geq cd$. Then, from the previous analysis we infer that inequality \eqref{reduction_first_prop} follows if we show the following -stronger- result (recall that $n$ is even): \begin{align}\label{improved_ineq}
\hbox{for all }\,x\in(0,\tfrac{1}{2}), \quad -\big(\Pi_n^2-\Pi_{n,0}^2\big)\geq 4x^2\Pi_{n,0}^2 \quad \hbox{and} \quad -\Sigma_{n-1}\geq \Sigma_{n-2} .
\end{align}
Notice that by gathering both inequalities we obtain \eqref{fprop_prime}. Hence, let us start by proving the first of them. In fact, by a direct re-arrangement of terms, it follows that the first inequality in \eqref{improved_ineq} is  equivalent to show that \[
(1-4x^2)\geq \Pi_{n,0}^{-2}\Pi_{n}^2=:\mathrm{R} \quad \hbox{ where } \quad 
\mathrm{R}=\prod_{k=1}^n\big(1-(k-\tfrac{1}{2})^{-2}x^2\big)^2.
\]
Now, on the one-hand, notice that the first factor in $\mathrm{R}$ (that is, the term associated with $k=1$) is given by $(1-4x^2)^2$. On the other hand, for all $k\in\{1,...,n\}$ and all $x\in(0,\tfrac{1}{2})$ we have \[
0< \big(1-(k-\tfrac{1}{2})^{-2}x^2\big)^2< 1.
\]
Thus, by plugging the latter inequality into the definition of $\mathrm{R}$, and by using the explicit form of the factor associated to $k=1$, it immediately follows that \[
\mathrm{R}\leq 1-4x^2.
\]
Now we focus on showing the second inequality in \eqref{improved_ineq}, that is, on showing $-\Sigma_{n-1}\geq \Sigma_{n-2}$. First of all notice that, for $x\in(0,\tfrac{1}{2})$, we can re-write these terms as \begin{align*}
\Sigma_{n-1}&=\Pi_n\cdot\sum_{k=1}^n\big(x^2-(k-\tfrac{1}{2})^2\big)^{-1},
\\ \Sigma_{n-2}&=\Pi_n\cdot\sum_{k=1}^{n-1}\sum_{j=k+1}^n\big(x^2-(k-\tfrac{1}{2})^2\big)^{-1}\big(x^2-(j-\tfrac{1}{2})^2\big)^{-1}.
\end{align*}
For the sake of simplicity, from now on we denote by $\Sigma_i^k$ the $k$-th term associated to $\Sigma_{i}$. More specifically, for the cases of $i=n-1$ and $i=n-2$, for each $k\in\{1,...,n\}$ we define
\begin{align*}
\Sigma_{n-1}^k&:= \big(x^2-(k-\tfrac{1}{2})^2\big)^{-1}\cdot\Pi_n 
\\ \Sigma_{n-2}^k&:=\big(x^2-(k-\tfrac{1}{2})^2\big)^{-1}\cdot\Pi_n\sum_{j=k+1}^n\big(x^2-(j-\tfrac{1}{2})^2\big)^{-1},
\end{align*}
where in the second case we assume $k\neq n$. Then, in order to show the second inequality in \eqref{improved_ineq}, it is enough to prove that, for each $k\in\{1,...,n-1\}$ and all $x\in (0,\tfrac{1}{2})$, \[
-\Sigma_{n-1}^k\geq\Sigma_{n-2}^k.
\]
In fact, once proving the latter inequality, it is enough to sum them all for all $k=1,...,n$, from where we conclude the desired result. Indeed, notice that, since $x\in(0,\tfrac{1}{2})$ we infer \[
\sum_{j=k+1}^n\big\vert x^2-(j-\tfrac{1}{2})^2\big\vert^{-1}\leq \sum_{j=k+1}^n\big\vert\tfrac{1}{4}-(j-\tfrac{1}{2})^2\big\vert^{-1}.
\] 
Therefore, recalling the following standard identity \[
\sum_{j=2}^n\big((j-\tfrac{1}{2})^2-\tfrac{1}{4}\big)^{-1}=\dfrac{n-1}{n}<1,
\]
by plugging the latter inequality into the definition of $\Sigma_{n-2}$ we deduce that  \[
\Sigma_{n-2}^k\leq -\big(x^2-(k-\tfrac{1}{2})^2\big)^{-1}\cdot\Pi_n=-\Sigma_{n-1}^k.
\]

\medskip

The case $n$ odd follows exactly the same lines (up to obvious modifications) and hence we omit it.
\end{proof}

Now, the following lemma give us the non-negativity of the sum of the third and fourth addend in the definition of $\mathbf{C}$.
\begin{lem}\label{mon_lem_2} Let $n\in\N$ with $n\geq 2$. Then, for all $x\in(-\tfrac{1}{2},\tfrac{1}{2})$ we have: \begin{align}\label{second_prop_ineq}
8\big(\Pi_n^2-\Pi_{n,0}^2\big)\Pi_n^2\Sigma_{n-1}\Sigma_{n-3}-16\big(\Pi_n^2-\Pi_{n,0}^2\big)\Pi_n^2\Sigma_{n-2}^2\geq0.
\end{align}
\end{lem}

\begin{proof}
In fact, similarly as before, we start by reducing the problem to an easier one. First of all notice that, for all $n\in\N$ with $n\geq 2$ and all $x\in(-\tfrac{1}{2},\tfrac{1}{2})$, we have \[
\big(\Pi_n^2-\Pi_{n,0}^2)\Pi_n^2\leq 0.
\]
Then, it follows that inequality \eqref{second_prop_ineq} is equivalent to prove that, for all $n\in\N$ with $n\geq 2$ and all $x\in(-\tfrac{1}{2},\tfrac{1}{2})$ it holds: \begin{align}\label{main_ineq_42}
2\Sigma_{n-2}^2\geq \Sigma_{n-1}\Sigma_{n-3}.
\end{align}
In this case, we shall not split the analysis into two different cases (comparing separately one factor from the left-hand side with another one from the right-hand side and then multiplying both inequalities). Instead, in this case it is easier to consider both factors at the same time. First of all, we rewrite both  sides of \eqref{main_ineq_42} as: \begin{align*}
\Sigma_{n-2}^2=\Pi_n^2\cdot\sum_{k=1}^{n-1}\sum_{j=1}^{n-1}\sum_{i=j+1}^n\sum_{\ell=k+1}^n \pi_{k,j,i\ell} \ \,\hbox{ and }\ \,  \Sigma_{n-1}\Sigma_{n-3}=\Pi_n^2\cdot\sum_{k=1}^n\sum_{j=1}^{n-2}\sum_{i=j+1}^{n-1}\sum_{\ell=i+1}^n\pi_{k,j,i,\ell},
\end{align*}
where, \[
\pi_{k,j,i,\ell}:=\big(x^2-(k-\tfrac{1}{2})^2\big)^{-1}\big(x^2-(j-\tfrac{1}{2})^2\big)^{-1}\big(x^2-(i-\tfrac{1}{2})^2\big)^{-1}\big(x^2-(\ell-\tfrac{1}{2})^2\big)^{-1}.
\]
Similarly as before, we shall compare each addend in the right-hand side of \eqref{main_ineq_42} to a corresponding (properly chosen) addend in the left-hand side. The idea of the proof is to show that each quadruple in the list defined by all possible combinations $(k,j,i,\ell)$ associated to the four sums in the right-hand side can be mapped to a proper permutation of itself, so that the resulting pair $(\sigma(k),\sigma(j),\sigma(i),\sigma(\ell))$ belongs to the list of possible combinations associated to the four sums in the left-hand side. Of course, the main difficulty in doing this is that both lists are not equivalent, and it is actually not possible to simply map them by using the identity map. However, by taking advantage of the factor $2$ in \eqref{main_ineq_42}, together with the fact that all terms in both sides are non-negative\footnote{Since each addend is composed by the multiplication of four simultaneously-non-positive factors.} for all $x\in(-\tfrac{1}{2},\tfrac{1}{2})$, we shall show that it is possible to map all of these elements from one list to the other one, where we shall use each element in the left-hand side list at most two times. Notice that the desired inequality follows once we prove that the previous procedure holds. 

\medskip

In fact, first of all notice that $\pi_{k,j,i,\ell}$ is invariant under permutations, that is, for any quadruple $(k,j,i,\ell)\in\{1,...,n\}^4$ we have \[
\pi_{k,j,i,\ell}=\pi_{\sigma(k),\sigma(j),\sigma(i),\sigma(\ell)},
\]
for any injective function $\sigma:\{k,j,i,\ell\}\to\{k,j,i,\ell\}$. Now, we define $\Gamma_{\mathrm{RHS}}$ and $\Gamma_{\mathrm{LHS}}$, the sets of indexes of all possible combinations associated with each side of \eqref{main_ineq_42}:
\begin{align*}
\Gamma_{\mathrm{RHS}}&:=\big\{(k,j,i,\ell)\in \{1,...,n\}^4: \ k\leq n-1,\ j<i<\ell\big\},
\\ \Gamma_{\mathrm{LHS}}&:= \big\{(k,j,i,\ell)\in\{1,...,n\}^4: \ k< \ell,\ j<i \big\}.
\end{align*}
We remark we have excluded the case $k=n$ in the definition of $\Gamma_{\mathrm{RHS}}$. The reason behind this is to be able to match (as a first case) both lists more easily (since $k=n$ is not allowed in the left-hand side of \eqref{main_ineq_42}). We shall address this exceptional case at the end of the proof. In this sense, one important (yet trivial) observation is that, the cardinality of the set of all possible combinations associated to each side is given by \[
\big\vert\Gamma_{\mathrm{LHS}}\big\vert= \dfrac{n^2(n^2-1)}{4},\quad \hbox{and} \quad \big\vert\Gamma_{\mathrm{RHS}}^{+n}\big\vert=\dfrac{n^2(n^2-3n+2)}{6},
\]
where $ \Gamma_{\mathrm{RHS}}^{+n}:=\big\{(k,j,i,\ell)\in\{1,...,n\}^4:\ j<i<\ell\big\}$. Additionally, for all $n\geq 2$, we have $\vert\Gamma_{\mathrm{LHS}}\vert\geq \vert\Gamma_{\mathrm{RHS}}^{+n}\vert$. Said that, as remarked before, we shall split the set of indexes given by the right-hand side and map them into the set of indexes appearing in the left-hand side. Having all of this in mind, we split the analysis into three main steps.

\medskip

\textbf{Case $j\geq k-1$:} In this case, by the definition of both sets $\Gamma_{\mathrm{RHS}}$ and $\Gamma_{\mathrm{LHS}}$ we trivially have that: \[
\hbox{if } \, (k,j,i,\ell)\in\Gamma_{\mathrm{RHS}}, \, \hbox{ then } \ (k,j,i,\ell)\in\Gamma_{\mathrm{LHS}}.
\]
In fact, it is enough to notice that, if $(k,j,i,\ell)\in\Gamma_{\mathrm{RHS}}$, then, by the definition of $\Gamma_{\mathrm{RHS}}$ it follows \[
\ell\geq i+1\geq j+2\geq k+1,
\]
where we have used the fact that $j\geq k-1$ to obtain the latter inequality. Hence, we deduce that, in this case, it is enough to map $(k,j,i,\ell)$ to itself.

\medskip

\textbf{Case $k\geq j+2$:} Let us consider any quadruple $(k,j,i,\ell)\in\Gamma_{\mathrm{RHS}}$ with $k\geq j+2$. We split the analysis into three different sub-cases.
\begin{itemize} 
\item Case $\ell\geq k+1$. Again, since $\ell\geq k+1$, by definition of $\Gamma_{\mathrm{LHS}}$ it immediately follows that \[
(k,j,i,\ell)\in\Gamma_{\mathrm{LHS}}.
\]
\item Case $\ell=k$. In this case we permute the coordinates in the following way: \[
(k,j,i,k)\mapsto (\widetilde{k},\widetilde{j},\widetilde{i},\widetilde{\ell}) \quad \hbox{where} \quad \widetilde{k}=j,\ \widetilde{j}=i, \ \widetilde{i}=k, \ \widetilde{\ell}=k.
\]
With these definitions it is not hard to see that $(\widetilde{k},\widetilde{j},\widetilde{i},\widetilde{\ell})\in\Gamma_{\mathrm{LHS}}$. In fact, it is enough to notice that, on the one hand, by definition of $\Gamma_{\mathrm{RHS}}$ we have $\ell\geq i+1\geq j+2$, while on the other hand, by hypothesis $\ell=k$. Then, it follows that \[
k=\widetilde{\ell}\geq \widetilde{k}+2=j+2 \quad \hbox{ and } \quad k=\widetilde{i}\geq  \widetilde{j}+1=i+1.
\]
We point out that this quadruple $(\widetilde{k},\widetilde{j},\widetilde{i},\widetilde{\ell})$ has already been used in the first case ``$j\geq k-1$''. However, notice that, since $n-1\geq k=\widetilde{\ell}$, in the present situation we never reach a quadruple of the form $(\cdot,\cdot,\cdot,n)$. This fact shall be important at the end of the proof.

\item Case $\ell\leq k-1$. First of all notice that, if $\ell\leq k-1$ and $(k,j,i,\ell)\in\Gamma_{\mathrm{RHS}}$, it transpires that $k\geq j+3$. Consequently, in this case we permute the first and last entry of the quadruple: \[
(k,j,i,\ell)\mapsto \big(\widetilde{k},j,i,\widetilde{\ell}\big) \quad \hbox{ where } \quad \widetilde{k}=\ell \quad \hbox{and}\quad \widetilde{\ell}=k.
\]
Thus, with these definitions we obtain that $\widetilde{\ell}\geq \widetilde{k}+1$, and therefore, $\big(\widetilde{k},j,i,\widetilde{\ell}\big)\in\Gamma_{\mathrm{LHS}}$. Moreover, due to the fact that $\ell\geq j+2$, we infer that $\widetilde{k}\geq j+2$, and hence this quadruple has already been used in the first sub-case of the present case, that is, ``$k\geq j+2$, sub case $\ell\geq k+1$''. Of course, as remarked before, we have \[
\pi_{k,j,i,\ell}=\pi_{\widetilde{k},j,i,\widetilde{\ell}}.
\]
Finally, by the same reason as in the previous case, in the present situation we never reach any quadruple of the form $(\cdot,\cdot,\cdot, n)$.
\end{itemize}

\textbf{Case $k=n$:} By the previous procedure we have used (at most) two times many of the  quadruples on the list associated to the left-hand side. However, notice that we have used at most once any quadruple of the form $(\cdot,\cdot,\cdot,n)$. Now, if $(k,j,i,\ell)\in\Gamma_{\mathrm{RHS}}$, then $j<n$ and $i<\ell$. Thus, in this case, by taking advantage of the factor $2$ in \eqref{main_ineq_42} again, we map \[
(n,j,i,\ell)\mapsto (j,i,\ell,n)\in\Gamma_{\mathrm{LHS}}.
\] 
Therefore, we have mapped each addend of the right-hand side of \eqref{main_ineq_42}, to the ``same addend'' (numerically they are the same due to the invariance under permutations of $\pi_{k,j,i,\ell}$)  appearing in the left-hand side, where we are repeating each addend in the left-hand side at most two times. Finally, notice that, any other addends in the left-hand side that has not been used is non-negative. Hence, gathering all the previous analysis we conclude $
2\Sigma_{n-2}^2\geq \Sigma_{n-1}\Sigma_{n-3}$.
\end{proof}
Now, before going further and for the sake of simplicity, let us prove the following inequality which shall be useful to treat the remaining addends in $\mathcal{V}_n$.
\begin{lem} Let $n\in\N$ with $n\geq 2$. For all $x\in(-\tfrac{1}{2},\tfrac{1}{2})$ it holds:\begin{align}\label{improimpro}
-(\Pi_n^2-\Pi_{n,0}^2)&\geq 2x^2\Pi_{n,0}^2\sum_{k=1}^n\big(k-\tfrac{1}{2}\big)^{-2}
\\ & \quad -x^4\Pi_{n,0}^2\left(\sum_{k=1}^n\big(k-\tfrac{1}{2}\big)^{-4}+4\sum_{k=1}^{n-1}\sum_{j=k+1}^n\big(k-\tfrac{1}{2}\big)^{-2}\big(j-\tfrac{1}{2}\big)^{-2}\right).\nonumber
\end{align}
\end{lem}

\begin{proof}
Intuitively, the right-hand side of \eqref{improimpro} corresponds to the first two terms in the expansion of the left-hand side. Moreover, it is not difficult to see (by using the fact that $x\in(-\tfrac{1}{2},\tfrac{1}{2})$) that the right-hand side in \eqref{improimpro} is always non-negative. Now, for the sake of clarity let us start by proving inequality \eqref{improimpro} for the case $n=2$. In fact, in this case the left-hand side becomes \[
-(\Pi_2^2-\Pi_{2,0}^2)=-x^8+5x^6-\tfrac{59}{8}x^4+\tfrac{45}{16}x^2.
\]
On the other hand, when $n=2$ both terms in the right-hand can be simply computed as: \[
2x^2\Pi_{2,0}^2\sum_{k=1}^2\big(k-\tfrac{1}{2}\big)^{-2}=\tfrac{45}{16}x^2, 
\]
and \[
-x^4\Pi_{2,0}^2\sum_{k=1}^2\big(k-\tfrac{1}{2}\big)^{-4}-x^4\Pi_{2,0}^2\sum_{k=1}^1\sum_{j=2}^2\big(k-\tfrac{1}{2}\big)^{-2}\big(j-\tfrac{1}{2}\big)^{-2}=-\tfrac{59}{8}x^4.
\]
Therefore, by noticing that $-x^8+5x^6\geq0$ for all $x\in(-\tfrac{1}{2},\tfrac{1}{2})$ we conclude the case $n=2$. For the general case, after trivial rearrangements, we can rewrite inequality \eqref{improimpro} as: \begin{align}\label{equi}
&\Pi_{n,0}^2-2x^2\Pi_{n,0}^2\sum_{k=1}^n\big(k-\tfrac{1}{2}\big)^{-2}
\\ & \quad +x^4\Pi_{n,0}^2\left(\sum_{k=1}^n\big(k-\tfrac{1}{2}\big)^{-4}+4\sum_{k=1}^{n-1}\sum_{j=k+1}^n\big(k-\tfrac{1}{2}\big)^{-2}\big(j-\tfrac{1}{2}\big)^{-2}\right)\geq \Pi_n^2 \nonumber
\end{align}
Since we have already proved the case $n=2$, from now on we shall assume that $n\geq 3$. Hence, it is enough to prove \eqref{equi}. In order to do that, we express $\Pi_n^2$ as: \begin{align}\label{expansion_pi}
\Pi_n^2=a_0^2-a_1^2x^2+a_2^2x^4\mp...+a_{2n}^2x^{4n}.
\end{align}
By explicit computations it is not difficult to check that, for any $n\in\N$ with $n\geq 3$ we have \begin{align*}
a_0^2&=\Pi_{n,0}^2, \qquad a_1^2=2\Pi_{n,0}^2\sum_{k=1}^n \big(k-\tfrac{1}{2}\big)^{-2},
\\ a_2^2&=\Pi_{n,0}^2\left(\sum_{k=1}^n\big(k-\tfrac{1}{2})^{-4}+4\sum_{k=1}^{n-1}\sum_{j=k+1}^n\big(k-\tfrac{1}{2}\big)^{-2}\big(j-\tfrac{1}{2}\big)^{-2}\right).
\end{align*}
Therefore, by plugging these identities into \eqref{equi}, and after direct cancellations we deduce that the problem is equivalent to prove: \[
0\geq -a_3^2x^6+a_4^2x^8\mp....+a_{2n}^2x^{4n}=:g(x),
\]
where $a_3^2,...,a_{2n}^2$ are the coefficients appearing in \eqref{expansion_pi}. Now, we group the addends in the definition of $g(x)$ into pairs of ``easier'' addends as:\[
g(x)=(-a_3^2x^6+a_4^2x^8)+(-a_5^2x^{10}+a_6^2x^{12})+...+(-a_{2n-1}^2x^{4n-1}+a_{2n}^2x^{4n}).
\]
Now, we claim that for all $m\in\{2,...,n\}$ the following holds:
\begin{align}\label{equiv_impro_proof}
4a_{2m-1}^2\geq a_{2m}^2.
\end{align}
Notice that, if we assume that the claim is true for the moment, then, gathering the latter inequality together with the fact that $x\in (-\tfrac{1}{2},\tfrac{1}{2})$ we would infer that, for all $m\in\{2,...,n\}$, \[
0\geq -a_{2m-1}^2x^{4m-2}+a_{2m}x^{4m}=x^{4m-2}(-a_{2m-1}^2+a_{2m}x^2) \quad \hbox{ for all } \, x\in (-\tfrac{1}{2},\tfrac{1}{2}).
\]
Clearly this would conclude the proof of inequality \eqref{equi}, and hence the proof of the lemma. Now, for the sake of clarity let us start by explicitly writing the first two cases ($a_3$ and $a_4$). In fact, by explicit computations we have: \begin{align*}
a_3^2&=2\Pi_{n,0}^2\sum_{i_1=1}^n\sum_{\substack{i_2=1 \\ i_2\neq i_1}}^n\big(i_1-\tfrac{1}{2}\big)^{-4}\big(i_2-\tfrac{1}{2}\big)^{-2}
\\ & \quad +8\Pi_{n,0}^2\sum_{i_1=1}^{n-2}\sum_{i_2=i_1+1}^{n-1}\sum_{i_3=i_2+1}^n\big(i_1-\tfrac{1}{2}\big)^{-2}\big(i_2-\tfrac{1}{2}\big)^{-2}\big(i_3-\tfrac{1}{2}\big)^{-2},
\\ a_4^2&=\Pi_{n,0}^2\sum_{i_1=1}^{n-1}\sum_{i_2=i_1+1}^n\big(i_1-\tfrac{1}{2}\big)^{-4}\big(i_2-\tfrac{1}{2}\big)^{-4}
\\ & \quad + 4\Pi_{n,0}^2\sum_{i_1=1}^n\sum_{\substack{i_2=1 \\ i_2\neq i_1}}^n\sum_{\substack{i_3=i_2+1 \\ i_3\neq i_1}}^n\big(i_1-\tfrac{1}{2}\big)^{-4}\big(i_2-\tfrac{1}{2}\big)^{-2}\big(i_3-\tfrac{1}{2}\big)^{-2}
\\ & \quad +16\Pi_{n,0}^2\sum_{i_1=1}^{n-3}\sum_{i_2=i_1+1}^{n-2}\sum_{i_3=i_2+1}^{n-1}\sum_{i_4=i_3+1}^{n} \big(i_1-\tfrac{1}{2}\big)^{-2}\big(i_2-\tfrac{1}{2}\big)^{-2}\big(i_3-\tfrac{1}{2}\big)^{-2}\big(i_4-\tfrac{1}{2}\big)^{-2}.
\end{align*}
Now, for the general case we distinguish two different cases, each of which is simultaneously composed by two different sub-cases ($a_{2m-1}$ and $a_{2m}$). The main difference between these inner sub-cases comes from the fact that $2m-1$ is always odd and $2m$ always even.

\medskip

\textbf{Case $m\leq n$:} By basic combinatorial arguments it is not difficult to see that $a_{2m-1}$ can be explicitly written as the sum of $m$ different type of terms. In fact, in order to do that let us start by describing the set of indexes that define each of these terms. Indeed, for $k=1,...,m$ we define the sets $\Gamma_1^{2m-1},...,\Gamma_m^{2m-1}$ as: 
\begin{align*}
\Gamma_k^{2m-1}&:=\big\{(i_1,...,i_{m+k-1})\in\N^{m+k-1}: \ 1\leq i_1<...<i_{m-k}\leq n,
\\ & \qquad \ 1\leq i_{m-k+1}<...<i_{m+k-1}\leq n, \ \, i_{m-k+1},...,i_{m+k-1}\notin\{i_1,...,i_{m-k}\} \big\}.
\end{align*}
In other words, each $\Gamma^m_k$ is composed by two different types of indexes. First we have $(m-k)$-indexes which are internally ordered. Then, we have the remaining $(2k-1)$-indexes which are simultaneously internally ordered (and they never coincide). Intuitively, the first $(m-k)$ indexes shall be associated to the factors with power $-4$ in the sums below, while the remaining $(2k-1)$ indexes shall be associated to the factors with power $-2$.
Then, taking advantage of the definition of $\Gamma^{2m-1}_k$ we can write $a_{2m-1}$ as\footnote{This can be seen as having two different copies of a list of $n$ elements. If we choose $2m-1$ elements out of the ``extended list'' of $2n$ elements, each element can be chosen in two different ways. The $m$ different types of terms (and the motivation for defining these $\Gamma_k^{2m-1}$) are associated to the number of repeated elements we choose.}:
\begin{align}\label{def_a_2m_1}
a_{2m-1}&= \Pi_{n,0}^2\sum_{k=1}^m 2^{2j-1}\sum_{(i_1,...,i_{m+k-1})\in\Gamma_k^{2m-1}}\big(i_1-\tfrac{1}{2}\big)^{-4}...\, \big(i_{m-k}-\tfrac{1}{2}\big)^{-4}\times \nonumber
\\ & \qquad \times \big(i_{m-k+1}-\tfrac{1}{2}\big)^{-2}...\, \big(i_{m+k-1}-\tfrac{1}{2}\big)^{-2}.
\end{align}
A few words to clarify the limit cases: Notice that, when $k=m$, the inner sum is composed only by terms with power $-2$, while in the case $k=1$ there is only one factor with power $-2$ (exactly as in the definition of the sets $\Gamma_m^m$ and $\Gamma_1^m$ respectively).

\medskip

Now for $a_{2m}$, it is not difficult to see that $a_{2m}$ can be expressed as the sum of $m+1$ different types of terms. Similarly as before, we start by describing the set of indexes for each of these sums. In fact, for $k=0,...,m$ we define the sets $\Gamma_{0}^{2m},...,\Gamma_m^{2m}$ as:
\begin{align*}
\Gamma_k^{2m}&:=\big\{(i_1,...,i_{m+k})\in\N^{m+k}: \ 1\leq i_1<...<i_{m-k}\leq n,
\\ & \qquad \ 1\leq i_{m-k+1}<...<i_{m+k}\leq n, \ \, i_{m-k+1},...,i_{m+k}\notin\{i_1,...,i_{m-k}\}\big\}.
\end{align*}
We emphasize that in this case $k$ starts at $k=0$ (in contrast with the previous case). Of course, in the present case as well as in the previous one above, whenever a set of indexes becomes empty, then the corresponding constraint does not exist. For example, in the latter definition, when $k=0$, the inequality \[
1\leq i_{m-k+1}<...<i_{m+k}\leq n,
\]
always holds (it is vacuously true since $i_{m-k+1}$ does not exists). Then, by taking advantage of the definition of $\Gamma_k^{2m}$, we can express $a_{2m}$ as: \begin{align}\label{def_a_2m}
a_{2m}&:= \Pi_{n,0}^2\sum_{k=0}^m 2^{2k}\sum_{(i_1,...,i_{m+k})\in\Gamma_k^{2m}}\big(i_1-\tfrac{1}{2}\big)^{-4}...\, \big(i_{m-k}-\tfrac{1}{2}\big)^{-4}\times\nonumber 
\\ & \qquad \ \times \big(i_{m-k+1}-\tfrac{1}{2}\big)^{-2}...\, \big(i_{m+k}-\tfrac{1}{2}\big)^{-2}.
\end{align}
Finally, it is not too difficult to prove \eqref{equiv_impro_proof} by using the previous expressions and by recalling the following standard (but useful) identities:
\begin{align}\label{tech_ident_proof}
\sum_{i=2}^\infty \big(i-\tfrac{1}{2}\big)^{-2}=\tfrac{\pi^2-8}{2}<1 \quad \hbox{ and }\quad \sum_{i=2}^\infty \big(i-\tfrac{1}{2}\big)^{-4}=\tfrac{\pi^4-96}{6}<1.
\end{align}
In fact, having all of the previous identities and definitions at hand, the idea is to notice that, except for the case $i=1$, all factors $(i-\tfrac{1}{2})^{-1}$ are smaller than $1$. Even more, as the previous identities show, their square and fourth-power are summable, and their sums are smaller than $1$. That motivates us to compare the sums over the set $\Gamma_k^{2m-1}$ with respect to the one associated to $\Gamma_{k}^{2m}$. We point out that, for each $k=1,...,m$ (we skip the case $k=0$ for the moment), the vectors in $\Gamma_{k}^{2m}$ have exactly one more coordinate than the ones in $\Gamma_k^{2m-1}$. Then, if $i_{m+k-1}\neq n$, for any $(i_1,...,i_{m+k-1})\in\Gamma_k^{2m-1}$ we define the restriction set \begin{align*}
\Gamma_k^{2m}[i_1,...,i_{m+k-1}]&:=\big\{(j_1,...,j_{m+k})\in\N^{m+k}: \ j_1=i_1,...,\ j_{m+k-1}=i_{m+k-1}, 
\\ & \qquad \ (j_1,...,j_{m+k})\in \Gamma_k^{2m} \big\}.
\end{align*}
Then, by using \eqref{tech_ident_proof} it immediately follows that \begin{align}\label{restricted_i}
&\big(i_1-\tfrac{1}{2}\big)^{-4}...\big(i_{m-k}-\tfrac{1}{2}\big)^{-4}\big(i_{m-k+1}-\tfrac{1}{2}\big)^{-2}...\big(i_{m+k-1}-\tfrac{1}{2}\big)^{-2}\geq 
\\ & \quad \geq \sum_{(j_1,...,j_{m+k})\in \Gamma_k^{2m}[i_1,...,i_{m+k-1}]}\big(j_1-\tfrac{1}{2}\big)^{-4}...\big(j_{m-k}-\tfrac{1}{2}\big)^{-4}\big(j_{m-k+1}-\tfrac{1}{2}\big)^{-2}...\big(j_{m+k}-\tfrac{1}{2}\big)^{-2}\nonumber
\end{align} 
Notice that, by gathering inequality \eqref{restricted_i} for all $(i_1,...,i_{m+k-1})\in\Gamma_{k}^{2m-1}$ with $i_{m+k-1}\neq n$ we obtain exactly the sum over $\Gamma_m^{2k}$ on the right-hand side. However, the resulting sum in the left-hand side produced by the previous procedure is strictly smaller than the sum over all possible indexes in $\Gamma_{m}^{2m-1}$ since we have never used any index with $i_{m+k-1}=n$. Finally, notice that for fixed $k\in\{1,...,m\}$, the corresponding sum over $\Gamma_k^{2m}$ in the definition of $a_{2m}$ (see \eqref{def_a_2m} above) has an extra $2$ factor (extra with respect to the same term in $a_{2m-1}$, see \eqref{def_a_2m_1} above). Therefore, taking into account this extra multiplicative factor $2$ on the right-hand side, the analysis above ensure us that \begin{align*}
2a_{2m-1}&\geq  \Pi_{n,0}^2\sum_{k=1}^m 2^{2k}\sum_{(i_1,...,i_{m+k})\in\Gamma_k^{2m}}\big(i_1-\tfrac{1}{2}\big)^{-4}...\, \big(i_{m-k}-\tfrac{1}{2}\big)^{-4}\times
\\ & \quad \ \times \big(i_{m-k+1}-\tfrac{1}{2}\big)^{-2}...\, \big(i_{m+k}-\tfrac{1}{2}\big)^{-2}.
\end{align*}
Finally, we shall use the remaining $2a_{2m-1}$ in the left-hand side of \eqref{equiv_impro_proof} to bound the sum associated to $\Gamma_0^{2m}$. In fact, it is easy to see from the definitions of $\Gamma_i^{2m}$ that $\Gamma_0^{2m}\subseteq \Gamma_1^{2m-1}$. Then, for any $(i_1,...,i_m)\in\Gamma_0^{2m}$, since the last entry always satisfies $(i_{m}-\tfrac{1}{2})^{-1}<1$, we infer \begin{align}\label{restricted_i_2}\big(i_1-\tfrac{1}{2}\big)^{-4}...\big(i_{m-1}-\tfrac{1}{2}\big)^{-4}\big(i_{m}-\tfrac{1}{2}\big)^{-2}\geq  \big(i_1-\tfrac{1}{2}\big)^{-4}...\big(i_{m}-\tfrac{1}{2}\big)^{-4}.
\end{align} 
Gathering inequality \eqref{restricted_i_2} associated to all possible $(i_1,...,i_m)\in\Gamma_0^{2m}$  we obtain that \[
2a_{2m-1}\geq  \Pi_{n,0}^2\sum_{(i_1,...,i_{m})\in\Gamma_0^{2m}}\big(i_1-\tfrac{1}{2}\big)^{-4}...\, \big(i_{m}-\tfrac{1}{2}\big)^{-4},
\]
and therefore $4a_{2m-1}^2\geq a_{2m}^2$, which finish the proof of the lemma for the case $m\leq n$. The case $m> n$ follows exactly the same lines (up to obvious modifications) and hence we omit it.
\end{proof}

With this lemma at hand we are able to handle the remaining terms in $\mathcal{V}_n$, that is, the sum of $\mathbf{A}$ with the first addends in $\mathbf{B}$ and $\mathbf{C}$. We recall that, in the definition of $\mathcal{V}_n$, the factors $\mathbf{A}$, $\mathbf{B}$ and $\mathbf{C}$ are multiplied by $x^0$, $x^2$ and $x^4$ respectively. Notice that the next proposition concludes the of the non-negativity of $\mathcal{V}_n$ in $(-\tfrac{1}{2},\tfrac{1}{2})$.

\begin{prop}\label{mon_prop_1} Let $n\in\N$ with $n\geq 2$. For all $x\in(-\tfrac{1}{2},\tfrac{1}{2})$ it holds:
\begin{align}\label{main_ineq_mon}
-\big(\Pi_n^2-\Pi_{n,0}^2\big)\Pi_n^2\Sigma_{n-1}^2+2x^2\Pi_{n,0}^2\Pi_n\Sigma_{n-1}^3+4x^4\Pi_{n,0}^2\Sigma_{n-1}^4\geq0.
\end{align}
\end{prop}

\begin{proof}
In fact, first of all, by factorizing by $\Sigma_{n-1}^2$ we infer that inequality \eqref{main_ineq_42} is equivalent to proving that, for all $x\in(-\tfrac{1}{2},\tfrac{1}{2})$ the following holds: \begin{align}\label{ineq_HH}
\mathrm{H}:=-\big(\Pi_n^2-\Pi_{n,0}^2\big)\Pi_n^2+2x^2\Pi_{n,0}^2\Pi_n\Sigma_{n-1}+4x^4\Pi_{n,0}^2\Sigma_{n-1}^2\geq0.
\end{align}
On the other hand, notice that by using inequality  \eqref{improimpro} we have  \begin{align*}
\mathrm{H}(x)&\geq  x^2\Pi_{n,0}^2\Bigg(2\Pi_n^2\sum_{k=1}^n\big(k-\tfrac{1}{2}\big)^{-2}+2\Pi_n\Sigma_{n-1}+4x^2\Sigma_{n-1}^2-x^2\Pi_n^2\sum_{k=1}^n\big(k-\tfrac{1}{2}\big)^{-4}
\\ & \quad  -4x^2\Pi_n^2\sum_{k=1}^{n-1}\sum_{j=k+1}^n\big(k-\tfrac{1}{2}\big)^{-2}\big(j-\tfrac{1}{2}\big)^{-2}\Bigg)
\\ &  = x^2\Pi_{n,0}^2\Pi_n^2\sum_{k=1}^n\Bigg(2\big(k-\tfrac{1}{2})^{-2}-2\big((k-\tfrac{1}{2})^2-x^2\big)^{-1}-x^2\big(k-\tfrac{1}{2}\big)^{-4}
\\ & \quad  +4x^2\big((k-\tfrac{1}{2})^2-x^2\big)^{-1}\sum_{j=1}^n\big((j-\tfrac{1}{2})^2-x^2\big)^{-1}
\\ & \quad -4x^2\big(k-\tfrac{1}{2}\big)^{-2}\sum_{j=k+1}^n\big(j-\tfrac{1}{2}\big)^{-2}\Bigg)=:x^2\Pi_{n,0}^2\Pi_n^2\sum_{k=1}^n \Delta_k,
\end{align*}
where the last sum before the last equality (the one indexed by $j=k+1,...,n$) must to be understood as zero when $k=n$. Then, in order to conclude inequality \eqref{ineq_HH} it is enough to show that $\Delta_k(x)\geq0$ for all $k\in\{1,...n\}$ and all $x\in(-\tfrac{1}{2},\tfrac{1}{2})$. In fact, first of all, recalling the first inequality in \eqref{tech_ident_proof} we deduce that, for all $n\in\N$ with $n\geq 2$, and any $k\in\{1,...,n\}$,  \[
-\sum_{j=k+1}^n\big(j-\tfrac{1}{2}\big)^{-2}\geq -\sum_{j=2}^\infty\big(j-\tfrac{1}{2}\big)^{-2}>-1 \quad \hbox{ and }\quad \sum_{j=1}^n\big(j-\tfrac{1}{2}\big)^{-2}\geq 4,
\]
%Moreover, noticing that for $j=1$ we have $(j-\tfrac{1}{2})^{-2}=4$, it trivially follows that \[
%\sum_{j=1}^n\big(j-\tfrac{1}{2}\big)^{-2}\geq 4.
%&\]
where the latter inequality simply follows by noticing that, when $j=1$, the first factor in the sum above $(j-\tfrac{1}{2})^{-2}=4$. Thus, by plugging the last two inequalities into the definition of $\Delta_k$, it follows \begin{align*}
\Delta_k&\geq 2\big(k-\tfrac{1}{2})^{-2}-2\big((k-\tfrac{1}{2})^2-x^2\big)^{-1}-x^2\big(k-\tfrac{1}{2}\big)^{-4}
\\ & \quad  +16x^2\big((k-\tfrac{1}{2})^2-x^2\big)^{-1}-4x^2\big(k-\tfrac{1}{2}\big)^{-2}=:\widetilde{\Delta}_k.
\end{align*}
Then, since $\widetilde{\Delta}_k(x)$ is even and $\widetilde{\Delta}_k(0)=0$, we infer that, to prove the non-negativity of each $\Delta_k$, it is enough to prove that $\tfrac{d}{dx}\widetilde{\Delta}_k(x)\geq 0$ for all $x\in (0,\tfrac{1}{2})$ and all $k\in\{1,...,n\}$. In fact, first of all, for the sake of simplicity let us start by re-writing $\widetilde{\Delta}_k$ as \begin{align*}
\widetilde{\Delta}_k&:=2\big(k-\tfrac{1}{2})^{-2}+2(8x^2-1)\big((k-\tfrac{1}{2})^2-x^2\big)^{-1}-x^2\big(k-\tfrac{1}{2}\big)^{-2}\big(4+\big(k-\tfrac{1}{2}\big)^{-2}\big). 
\end{align*}
Then, by direct computations we get:
\begin{align*}
\dfrac{d}{dx}\widetilde{\Delta}_k(x)&=\dfrac{4x\big(8(k-\tfrac{1}{2})^2-1\big)}{\big((k-\tfrac{1}{2})^2-x^2\big)^2}-\dfrac{2x\big(4(k-\tfrac{1}{2})^{2}+1\big)}{\big(k-\tfrac{1}{2})^4}=:\dfrac{\mathrm{A}}{\mathrm{B}}-\dfrac{\mathrm{C}}{\mathrm{D}}.
\end{align*}
Notice that $\mathrm{A},\,\mathrm{B},\,\mathrm{C},\,\mathrm{D}\geq0$ for $x\in (0,\tfrac{1}{2})$. Hence, it is enough to prove that \[
\mathrm{A}\geq \mathrm{C}\quad \hbox{and}\quad \mathrm{B}^{-1}\geq \mathrm{D}^{-1}.
\]
We point out that inequality $\mathrm{B}^{-1}\geq \mathrm{D}^{-1}$ follows directly. In fact, recalling that $x\in(-\tfrac{1}{2},\tfrac{1}{2})$, we deduce \[
\dfrac{1}{\mathrm{B}}=\dfrac{1}{\big((k-\tfrac{1}{2})^2-x^2\big)^2}\geq \dfrac{1}{(k-\tfrac{1}{2})^4}=\dfrac{1}{\mathrm{D}}.
\]
Then, it only remains to prove that $\mathrm{A}\geq \mathrm{C}$. In fact, by direct computations it immediately follows that, for $x\in (0,\tfrac{1}{2})$ and $k\in\{1,...,n\}$, we have \[
\mathrm{A}-\mathrm{C}=6x\big(4(k-\tfrac{1}{2})^2-1\big)=24kx\big(k-1\big)\geq0.
\]
Therefore, $\tfrac{d}{dx}\widetilde{\Delta}_k(x)\geq 0$ for all $x\in(0,\tfrac{1}{2})$ and all $k\in \{1,...,n\}$, which concludes the proof.
\end{proof}

\begin{proof}[End of the proof of Theorem \ref{thm_monotonicity}]

We start by pointing out that, by gathering Lemma \ref{mon_lem_1} and \ref{mon_lem_2} together with Proposition \ref{mon_prop_1}, we conclude that, for all $x\in (-\tfrac{1}{2},\tfrac{1}{2})$ it holds: \[
-\dfrac{d^2}{dx^2}\dfrac{V_{n}(x)-V_n(0)}{(V_{n}'(x))^2}=\dfrac{3V_{n}''\big( (V_{n}')^2-2(V_{n}-V_n(0))V_{n}''\big)+2(V_{n}-V_n(0))V_{n}'V_{n}'''}{(V_{n}')^4}\geq 0.
\]
Additionally, it is not difficult to see from the inequalities above that, whenever $x\neq 0$, the latter inequality holds strictly\footnote{It is enough to notice that, for example, some of the previous lemmas are proven by showing that certain even function $f(x)$ satisfies $f(0)=0$ with $f'(x)>0$ for $x\in(0,\tfrac{1}{2})$.}. However, due to the factor $(V_n')^4$, the latter quantity might have a singularity at $x=0$. Thus, it only remains to prove that, when $x$ goes to zero, the latter quantity is well-defined and strictly positive. Notice that this shall conclude the proof of the theorem by applying the main result in \cite{Chi}. Indeed, let us start by pointing out that, from the explicit formula of $V_{n}'$ in \eqref{comp_v_x}, we infer that $(V_{n}')^4=o(x^4)$. Then, we must first prove that $\mathcal{V}_{n}$ is also $o(x^4)$. In order to do this, we start by recalling that $\tfrac{1}{96}\mathcal{V}_n=\mathbf{A}+\mathbf{B}x^2+\mathbf{C}x^4$, where $\mathbf{A}$, $\mathbf{B}$ and $\mathbf{C}$ are given by: \begin{align*}
\mathbf{A}&:=-\big(\Pi_n^2-\Pi_{n,0}^2\big)\Pi_n^2\Sigma_{n-1}^2,
\\ \mathbf{B}&:=2\Pi_{n,0}^2\Pi_n\Sigma_{n-1}^3-4(\Pi_n^2-\Pi_{n,0}^2)\Pi_n^2\Sigma_{n-1}\Sigma_{n-2},
\\ \mathbf{C}&:=4\Pi_{n,0}^2\Sigma_{n-1}^4
+8\Pi_{n,0}^2\Pi_n\Sigma_{n-1}^2\Sigma_{n-2}+8(\Pi_n^2-\Pi_{n,0}^2)\Pi_n^2\Sigma_{n-1}\Sigma_{n-3}
\\ & \quad \ \, -16(\Pi_n^2-\Pi_{n,0}^2)\Pi_n^2\Sigma_{n-2}^2.
\end{align*}
Hence, in the sequel we seek to prove that $\lim_{x\to0}\tfrac{1}{x^4}\mathcal{V}_n(x)$ exists and is strictly positive. We split the analysis into two steps. First, we intend to prove that \begin{align}\label{second_order}
\lim_{x\to 0}\dfrac{1}{x^2}\big(\mathbf{A}+2x^2\Pi_{n,0}^2\Pi_n\Sigma_{n-1}^3\big)=0.
\end{align}
It is worth noticing that, the latter limit ensures us that the quantity inside the parenthesis in \eqref{second_order} behaves (at least) as $x^4$ near zero. In fact, first of all, recall that in the proof of \eqref{improimpro} we have already shown that \begin{align}\label{approx_thetas}
-(\Pi_n^2-\Pi_{n,0}^2)&=2x^2\Pi_{n,0}^2\sum_{k=1}^n\big(k-\tfrac{1}{2}\big)^{-2}\nonumber
\\ & \quad \,-x^4\Pi_{n,0}^2\left(\sum_{k=1}^n\big(k-\tfrac{1}{2}\big)^{-4}+4\sum_{k=1}^{n-1}\sum_{j=k+1}^n\big(k-\tfrac{1}{2}\big)^{-2}\big(j-\tfrac{1}{2}\big)^{-2}\right)+o(x^6)\nonumber
\\ & =:\Theta_{2}(x)+\Theta_{4}(x)+o(x^6),
\end{align}
where $\Theta_2(x)$ and $\Theta_4(x)$ denote the terms of order $x^2$ and $x^4$ respectively. Then, we gather the term in $\mathbf{A}$ associated with $\Theta_2$ with the first term appearing in $\mathbf{B}$. Specifically, we group \begin{align}\label{theta_2}
\Theta_2\Pi_n^2\Sigma_{n-1}^2+2x^2\Pi_{n,0}^2\Pi_n\Sigma_{n-1}^3&=2x^2\Pi_{n,0}^2\Pi_n\Sigma_{n-1}^2\left(\Pi_n\sum_{k=1}^n\big(k-\tfrac{1}{2}\big)^{-2}+\Sigma_{n-1}\right).
\end{align}
Then, it is enough to notice the following trivial identities: \begin{align}\label{lim_sigma_1}
&\lim_{x\to0}\Pi_n=(-1)^{n}\Pi_{n,0} \quad \hbox{and} \quad \lim_{x\to0}\Sigma_{n-1}=(-1)^{n-1}\Pi_{n,0}\sum_{k=1}^n\big(k-\tfrac{1}{2}\big)^{-2}.
\end{align}
By plugging the last two identities into the parenthesis in \eqref{theta_2} we conclude the proof of \eqref{second_order}. In particular, we infer that $
\Theta_2\Pi_n^2\Sigma_{n-1}^2+2x^2\Pi_{n,0}^2\Pi_n\Sigma_{n-1}^3=o(x^4)$. Similarly, now we gather the last term in $\mathbf{B}$ with the second one in $\mathbf{C}$. Specifically, we group (recall that the terms associated to $\mathbf{C}$ in $\mathcal{V}_n$ have an extra $x^2$ with respect to the ones in $\mathbf{B}$): \begin{align}\label{groupsII}
-4(\Pi_n^2-\Pi_{n,0}^2)\Pi_n^2\Sigma_{n-1}\Sigma_{n-2}+8x^2\Pi_{n,0}^2\Pi_n\Sigma_{n-1}^2\Sigma_{n-2}.
\end{align}
However, by using the second identity in \eqref{lim_sigma_1} and \eqref{approx_thetas} again, it is easy to see that \[
\lim_{x\to 0}\dfrac{1}{x^2}\big(4(\Pi_n^2-\Pi_{n,0}^2)\Pi_n^2\Sigma_{n-1}\Sigma_{n-2}-8x^2\Pi_{n,0}^2\Pi_n\Sigma_{n-1}^2\Sigma_{n-2}\big)=0.
\]
Hence, due to the extra $x^2$ factor, the terms appearing in $\mathcal{V}_n$ associated to \eqref{groupsII} are of order $o(x^6)$. It is worth to notice that, except for the first term in $\mathbf{C}$, all the remaining terms appearing in $\mathcal{V}_n$ that we have not treated so far are of order $o(x^6)$. Consequently, the problem is reduced to study the following limit: \[
\lim_{x\to 0}\dfrac{1}{x^4}\big((\Theta_2+\Theta_4)\Pi_n^2\Sigma_{n-1}^2+2x^2\Pi_{n,0}^2\Pi_n\Sigma_{n-1}^3+4x^4\Pi_{n,0}^2\Sigma_{n-1}^4\big).
\]
For the sake of simplicity let us start by some direct computations. In fact, on the one-hand we have: \begin{align*}
\Pi_n=\widehat{a}_0+\widehat{a}_1x^2+o(x^4), \quad  &\hbox{where} \quad \widehat{a}_1=\Sigma_{n-1}(0),
\\ \Sigma_{n-1}=\widetilde{a}_0+\widetilde{a}_1x^2+o(x^4), \quad &\hbox{where}\quad \widetilde{a}_1=2\Sigma_{n-2}(0).
\end{align*}
Thus, gathering \eqref{theta_2} with the last identities we infer \begin{align}\label{theta2_x4}
&\lim_{x\to0}\dfrac{1}{x^4}\big(\Theta_2\Pi_n^2\Sigma_{n-1}^2+2x^2\Pi_{n,0}^2\Pi_n\Sigma_{n-1}^3\big)=\nonumber
\\ & \qquad =2(-1)^n\Pi_{n,0}^3\Sigma_{n-1}^2(0)\left(\Sigma_{n-1}(0)\sum_{k=1}^n\big(k-\tfrac{1}{2}\big)^{-2}+2\Sigma_{n-2}(0)\right).
\end{align}
On the other hand, by taking limit directly in the definition of $\Theta_4(x)$ we obtain
\[
\lim_{x\to0}\dfrac{1}{x^4}\Theta_4(x)\Pi_n^2\Sigma_{n-1}^2=-\Pi_{n,0}^4\Sigma_{n-1}^2(0)\left(\sum_{k=1}^n\big(k-\tfrac{1}{2}\big)^{-4}+4\sum_{k=1}^{n-1}\sum_{j=k+1}^n\big(k-\tfrac{1}{2}\big)^{-2}\big(j-\tfrac{1}{2}\big)^{-2}\right).
\]
Finally, we gather the previous terms, that is, we gather the parenthesis in \eqref{theta2_x4} with the terms associated with the latter limit and the first term associated with $\mathbf{C}(x)$. More specifically, we group \begin{align*}
(\Theta_2+\Theta_4)\Pi_n^2\Sigma_{n-1}^2+2x^2\Pi_{n,0}^2\Pi_n\Sigma_{n-1}^3+4x^4\Pi_{n,0}^2\Sigma_{n-1}^4.
\end{align*}
By using the last two limits, identity \eqref{lim_sigma_1} once again, and then performing  some direct cancellations we obtain \begin{align*}
&\lim_{x\to0}\dfrac{1}{x^4}\left((\Theta_2+\Theta_4)\Pi_n^2\Sigma_{n-1}^2+2x^2\Pi_{n,0}^2\Pi_n\Sigma_{n-1}^3+4x^4\Pi_{n,0}^2\Sigma_{n-1}^4\right)=
\\ & \qquad = \Pi_{n,0}^4\Sigma_{n-1}^2(0)\sum_{k=1}^n\big(k-\tfrac{1}{2}\big)^{-4}+4(-1)^n\Pi_{n,0}^3\Sigma_{n-1}^2(0)\Sigma_{n-2}(0).
\end{align*}
Finally, since $(-1)^n\Sigma_{n-2}(0)>0$, we conclude the proof of the theorem.
\end{proof}

\medskip

\section{Spectral analysis \label{sec:SpecAna}}

In this section, we use the monotonicity of the period map $L$ with respect to the energy level $\beta$ to analyze
the spectrum of the linearized operator associated with the traveling wave
obtained in the previous section. From now on, with no loss of generality and in addition to the hypothesis in Theorem \ref{thm_monotonicity}, we shall assume $0\leq c<1$.

\subsection{Spectrum of the scalar linearized operator}

Our goal now is to study the spectral information associated to the scalar linear operator $\mathcal{L}_c$. Let us start by recalling that the odd traveling wave solution constructed in the last section satisfies
\begin{align}\label{trav_eq}
-\omega\phi_{c}'' +V_{n}'\left(\phi_{c}\right)=0,
\end{align}
where $\omega=1-c^2$. Then, the linearized operator $\mathcal{L}_{c}$ around $\phi_c$ is given by: \begin{align}\label{scalar_linear_op}
\mathcal{L}_{c}:=-\omega \partial_x^2+V_{n}''\left(\phi_{c}\right).
\end{align}
It is worth  noticing that $\mathcal{L}_c$ can be regarded as a bounded self-adjoint operator defined on 
$L^{2}\left(\mathbb{T}_{L}\right)$ with domain $H^{2}\left(\mathbb{T}_{L}\right)$. According to Oscillation Theorem, see Magnus-Winkler \cite{MaWi}, the spectrum of $\mathcal{L}_{c}$ is formed by a sequence of real numbers, bounded from below and going to $+\infty$ as $m$ goes to $+\infty$. More specifically, we can list the eigenvalues of $\mathcal{L}_c$ as
\[
\sigma_{0}<\sigma_{1}\leq\sigma_{2}<\sigma_{3}\leq\sigma_{4}<\cdots<\sigma_{2m-1}\leq\sigma_{2m}<\cdots.
\]
Moreover, the spectrum of $\mathcal{L}_{c}$ is also characterized by the number
of zeros of the corresponding eigenfunctions. Then, in order to analyze the (in)stability problem of traveling wave solutions, it is helpful to start studying the spectrum of $\mathcal{L}_c$ in more details.  We start by recalling two results of Floquet theory that shall be useful in the sequel.

\begin{thm}[\cite{N}, Theorem 2.2] \label{thm:Neves1} Let $p(x)$ be any $L$-periodic solution of $\mathcal{L}_{c}y=0$. Consider any other  linearly independent solution $y(x)$ such that the Wronskian $W(p,y)$ satisfies  
\[
W(p,y):=\det\left(\begin{matrix}
p & y
\\ p_{x} & y_{x}
\end{matrix}\right)=1.
\]
Then, $y(x+L)=y(x)+\theta p(x)$ for some constant only depending on $y$. In particular $y(x)$ is $L$-periodic if and only if $\theta=0$.
\end{thm}

\begin{rem}
We point out that the constant $\theta$ can be explicitly computed (see \cite{N}).
\end{rem}

%One can apply results above and information of $\theta $ to study the eigenvalue problems.

\begin{thm}[\cite{N}, Theorem 3.1]\label{thm:Neves2} Consider any eigenvalue  $\sigma_{k}$ of $\mathcal{L}_c$ with $k\geq1$, and its associated eigenfunction $\widetilde{p}(x)$. Let $\theta$ be the constant given in Theorem \ref{thm:Neves1} associated to the operator $\widetilde{\mathcal{L}}:=(\mathcal{L}_{c}-\sigma_{k})$ and $p(x)=\widetilde{p}(x)$. Then, $\sigma_{k}$ is a simple eigenvalue of $\mathcal{L}_c$ if and only if $\theta\neq0$. Furthermore, if $p(x)$ has $2m$-zeros in $[0,L)$, then the following holds: \[
\{\hbox{if } \theta<0, \hbox{ then } \sigma_{k}=\sigma_{2m-1}\} \qquad \hbox{and} \qquad  \{\hbox{if } \theta>0, \hbox{ then }\sigma_{k}=\sigma_{2m}\}.\]
\end{thm}

It is worthwhile to notice that, by differentiating the equation \eqref{trav_eq}, we infer that $\phi'_{c}$ belongs to $\ker(\mathcal{L}_{c})$, and hence zero is an eigenvalue of $\mathcal{L}_c$. Next we apply both theorems above to analyze the $0$ eigenvalue of
$\mathcal{L}_{c}$. By Theorem \ref{thm:Neves1}, for any solution to
$\mathcal{L}_{c}y=0$ linearly independent to $\phi_c'$ satisfying $W\left(\phi'_{c},y\right)=1$, one has
\begin{align}\label{eq_theta}
y\left(x+L\right)=y\left(x\right)+\theta\phi_{c}'\left(x\right),
\end{align}
for some constant $\theta$ only depending on $y$. Moreover, it is not difficult to see\footnote{From equation \eqref{trav_eq}, the oddness of the solution and the fact that $-(V(x)-V(0))$ is strictly increasing in $(0,\tfrac{1}{2})$, for example.} that $\phi'_{c}(x)$ has exactly two zeros $[0,L)$. Thus, by applying the Oscillation Theorem, we know that $0$ is either $\sigma_1$ or $\sigma_2$. To obtain more precise information of the eigenvalue $0$, by Theorem \ref{thm:Neves2}, we need to know $\theta$.  The next lemma connects $\theta$ and $\frac{\partial}{\partial\beta}L$ computed from Theorem \ref{thm_monotonicity}.
\begin{lem}
	\label{lem:thetaLbeta}
	Under our current hypothesis, we have the following relation between $\frac{d L}{d\beta}$ from Theorem \ref{thm_monotonicity}
	and $\theta$ from \eqref{eq_theta}:
\begin{align*}
	\theta=-\frac{\partial L}{\partial\beta}.
\end{align*}
\end{lem}

\begin{proof}
Our proof follows a similar spirit to that given in \cite{deNa}. However, notice that this latter one contains some typos that must to be corrected (see Section \ref{ext_nata} for further details). In fact, let us start by defining $\mu$ to be the unique solution to the problem
\begin{align*}
	\begin{cases}
	-\omega\mu''+V_{n}''(\phi_{c})\mu=0\\
	\mu(0)=0, \ \, \mu' (0)=\frac{1}{\phi_{c}'\left(0\right)} & .
	\end{cases}
\end{align*}
Notice that by  the definition of $\mu$ it immediately follows that $W(\phi'_{c},\mu)=1$. Then by Theorem \ref{thm:Neves1}, there is a constant $\theta$, only depending on $\mu$, such that
\begin{align*}
	\mu(x+L)=\mu(x)+\theta\phi_{c}'(x).
\end{align*}
Therefore, by evaluating the latter identity at $x=0$, recalling that by construction $\mu(0)=0$, we deduce that $\theta=(\phi_{c}'(0))^{-1}\mu(L)$. On the other hand, since $\phi_{c}$ is odd and periodic it follows that $\phi_{c}\left(0\right)=\phi_{c}\left(L\right)=0$. Thus, by differentiating the latter identity at $x=L$ with respect to $\beta$, we deduce:
\begin{align}\label{dLdb_2}
\phi_{c}'(L)\frac{\partial L}{\partial\beta}+\frac{\partial\phi_{c}}{\partial\beta}(L)=0 	\quad\implies 	\quad 	\frac{\partial L}{\partial\beta}=-\frac{1}{\phi_{c}'(0)}\frac{\partial\phi_{c}}{\partial\beta}(L),
\end{align}
where we have used the periodicity of the solution $\phi'_c(0)=\phi'_{c}(L)$. Finally, in order to obtain the relation between $\partial_{\beta}\phi_{c}$ and $\mu$, we start by recalling that, from Theorem \ref{thm_monotonicity} we know that for $L$ and $c$ fixed, there exist a unique $\beta(c)\in(0,E_\star)$ such that \begin{align}\label{ham_energy_beta_lvl}
\dfrac{1}{2}\phi_x^2(x)-\dfrac{1}{\omega}\big(V_n(\phi(x))-V_n(0)\big)=\beta.
\end{align} 
Hence, by differentiating the latter equation with respect to $\beta$, and then differentiating the resulting equation with respect to $x$ we obtain
\begin{align*}
	&\omega\phi_{c}''\partial_{\beta}\phi'_{c}+\omega\phi_{c}'\partial_{\beta}\phi_{c}''-V_{n}''\left(\phi_{c}\right)\phi_{c}'\partial_{\beta}\phi_{c}-V_{n}'\left(\phi_{c}\right)\partial_{\beta}\phi_{c}'\nonumber 
\\ & \quad =\partial_{\beta}\phi_{c}'\left(\omega\phi_{c}''-V_{n}'\left(\phi_{c}\right)\right)+\phi_{c}'\left(\omega\partial_{\beta}\phi_{c}''-V_{n}''\left(\phi_{c}\right)\partial_{\beta}\phi_{c}\right)=0.
	\end{align*}
Now, on the one-hand, by differentiating equation \eqref{trav_eq} with respect to $\beta$ we have
\begin{align}\label{phi_beta_2}
\omega\partial_{\beta}\phi_{c}''-V_{n}''\left(\phi_{c}\right)\partial_{\beta}\phi_{c}=0.\end{align}
On the other hand, evaluating identity \eqref{ham_energy_beta_lvl} at $x=0$, recalling that due to the oddness of the solution $\phi_c(0)=0$, we infer that:
\begin{align}\label{phibeta_origin}
	\partial_{\beta}\phi_{c}\left(0\right)=0\quad \text{ and }\quad \frac{1}{2}\phi_{c,x}^{2}\left(0\right)=\beta.
\end{align}
Finally, differentiating the second identity in \eqref{phibeta_origin} with respect to $\beta$ we infer that $\partial_{\beta}\phi_{c,x}(0)=(\phi_{c}'(0)\big)^{-1}$. Therefore, by gathering the latter identity with \eqref{phi_beta_2} and \eqref{phibeta_origin}, we conclude that $\partial_{\beta}\phi_{c}$ satisfies the same ODE as $\mu$ with the same initial data. By the uniqueness of the solution, it follows that $\partial_{\beta}\phi_{c}\equiv\mu$. Therefore, recalling that we have shown that $\theta=(\phi_c'(0))^{-1}\mu(L)$, together with the second identity in \eqref{dLdb_2}, we conclude $\theta=-\partial_{\beta}L$ as desired.
\end{proof}

By computations from Section \ref{existence}, see Theorem \ref{thm_monotonicity}, we know that $\theta=-\partial_{\beta}L<0$. Then as a direct application of Theorem \ref{thm:Neves2}, we can conclude the following spectral information
of $\mathcal{L}_{c}$.
\begin{prop}
	\label{prop:linearspec}
	Under the hypothesis of Theorem \ref{thm_monotonicity}, the linear operator $\mathcal{L}_{c}$ given by \eqref{scalar_linear_op} above, defined
	on $L^{2}\left(\mathbb{T}_{L}\right)$ with domain $H^{2}\left(\mathbb{T}_{L}\right)$, defines a bounded self-adjoint operator with exactly one negative
	eigenvalue, a simple eigenvalue at zero and the rest of its spectrum
	is positive, discrete and bounded away from zero.
\end{prop}

\subsection{Spectrum of the matrix operator}

Now we seek to use the previous spectral information for the scalar operator $\mathcal{L}_c$ to conclude related spectral properties associated to the so called linarized Hamiltonian. In fact, we start by pointing out that the equation solved by the periodic traveling wave solution $\vec{\phi}_c$ constructed in Section \ref{existence} can be re-written in terms of the conserved functionals $\mathcal{E}$ and $\mathcal{P}$ as
\[
\mathcal{E}'\big(\vec{\phi}_{c}\big)+c\mathcal{P}'\big(\vec{\phi}_{c}\big)=0,
\]
where $\mathcal{E}'$ and $\mathcal{P}'$ are the Frechet derivatives
of $\mathcal{E}$ and $\mathcal{P}$ in $H^{1}\left(\mathbb{T}_{L}\right)\times L^{2}\left(\mathbb{T}_{L}\right)$ respectively. Then, the linearized Hamiltonian around $\vec{\phi}_{c}$ is given by the matrix operator
\begin{equation}
\vec{\mathcal{L}}_{c}:=(\mathcal{E}''+c\mathcal{P}'')\big(\vec{\phi}_c\big)=\left(\begin{matrix}
-\partial_{xx}+V''_{n}\left(\phi_{c}\right) & -c\partial_{x}\\
c\partial_{x} & 1
\end{matrix}\right).\label{eq:mL}
\end{equation}
It is worthwhile to notice that $\vec{\mathcal{L}}_{c}$ can be regarded as a bounded self-adjoint operator defined on \[
\vec{\mathcal{L}}_c:H^{2}\left(\mathbb{T}_{L}\right)\times H^{1}\left(\mathbb{T}_{L}\right)\subset L^{2}\left(\mathbb{T}_{L}\right)\times L^{2}\left(\mathbb{T}_{L}\right)\to L^2(\T_L).
\]
Moreover, notice that with these definitions it immediately follows that $\vec{\phi}_{c,x}$ belongs to the kernel of $\vec{\mathcal{L}}_{c}$. On the other hand, the quadratic form $\mathcal{Q}_c$ associated to the matrix operator defined in  \eqref{eq:mL} is given by:
\begin{align}
\mathcal{Q}_{c}\big(\varphi_1, \varphi_2\big) & :=\big\langle \vec{\mathcal{L}}_{c}(\varphi_1, \varphi_2),(\varphi_1, \varphi_2)\big\rangle =\int_{\mathbb{T}_{L}}\big(\varphi_{1,x}^{2}+V_{n}''\left(\phi_{c}\right)\varphi_1^{2}+2c\varphi_{1,x}\varphi_2+\varphi_2^{2}\big)\,dx\nonumber \\
& %=\int_{\mathbb{T}_{L}}\left(1-c^{2}\right)\left(u_{x}\right)^{2}+V_{n}''\left(\phi%{c}\right)u^{2}+\left(cu_{x}-v\right)^{2}\,dx\label{eq:MQ}\\
=\int_{\mathbb{T}_{L}}\big(\omega\varphi_{1,x}^{2}+V_{n}''\left(\phi_{c}\right)\varphi_1^{2}\big)dx+\int_{\T_L}\big(c\varphi_{1,x}+\varphi_2\big)^{2}\,dx.\label{eq:MQ} 
\end{align}
It is worth noticing that, from the first integral term on the latter identity we recognize the scalar quadratic
form \begin{align}
Q_c(\varphi) :=\left\langle\mathcal{L}_c \varphi,\varphi\right\rangle=\omega\int_{\mathbb{T}_{L}}\varphi_{x}^{2}dx+\int_{\T_L}V_{n}''\left(\phi_{c}\right)\varphi^{2}dx,\label{eq:Q}
\end{align}
which is the quadratic form associated to the linear operator $\mathcal{L}_{c}$ in \eqref{scalar_linear_op}. The following lemma links the spectral information of $\mathcal{L}_c$ derived in the previous subsection with the one of $\vec{\mathcal{L}}_c$.

\begin{lem}
	\label{lem:matrixspe} Under the assumptions of Theorem \ref{thm_monotonicity}, the operator $\vec{\mathcal{L}}_{c}$ given in \eqref{eq:mL} defined on $L^{2}\left(\mathbb{T}_{L}\right)\times L^{2}\left(\mathbb{T}_{L}\right)$
	with domain $H^{2}\left(\mathbb{T}_{L}\right)\times H^{1}\left(\mathbb{T}_{L}\right)$
	defines a bounded self-adjoint operator with a unique negative eigenvalue. Furthermore, zero is the second eigenvalue, which is simple, and the rest of the spectrum is discrete and bounded away from zero.
\end{lem}

\begin{proof}
First, from Weyl's essential spectral Theorem, it follows that the essential spectra of $\vec{\mathcal{L}}_{c}$ is empty. Besides, by compact self-adjoint operator theory, $\vec{\mathcal{L}}_{c}$ has only point spectra. Next, we need to check the signs of eigenvalues. Recall that by Proposition
	\ref{prop:linearspec} we already know that $\mathcal{L}_{c}$ has exactly one simple negative eigenvalue and that zero is also a simple eigenvalue (with $\phi_{c}'$ as its associated eigenfunction). Let $\sigma_{0}$ be be the unique negative eigenvalue
of $\mathcal{L}_{c}$ with eigenfunction $Y_{0}$ (notice that $Y_0$ is even). Then, it immediately follow by the definition of $\sigma_0$ that 
\begin{align}\label{lambda0}
\mathcal{Q}_{c}\left(Y_{0},-cY_{0}'\right)=\sigma_{0}\int_{\mathbb{T}_{L}}Y_{0}^{2}\,dx+\int_{\mathbb{T}_{L}}\left(cY_{0}'-cY_{0}'\right)^{2}\,dx=\sigma_{0}\int_{\mathbb{T}_{L}}Y_{0}^{2}\,dx<0.
\end{align}
In the same fashion as in the previous section, by using Oscillation Theory we know we can list the eigenvalues of $\vec{\mathcal{L}}_{c}$ as \[
\widetilde{\sigma}_{0}<\widetilde{\sigma}_1\leq\widetilde{\sigma}_2<\widetilde{\sigma}_3\leq\widetilde{\sigma}_4<...<\widetilde{\sigma}_{2m-1}\leq\widetilde{\sigma}_{2m}<...
\]
Then by the using min-max principle (see for example \cite{RS}) and \eqref{lambda0}, we infer that $\widetilde{\sigma}_{0}<0$. Thus, in order to conclude it is enough to show that $\widetilde{\sigma}_{1}=0$ and $\widetilde{\sigma}_{2}>0$. In fact, for the sake of simplicity let us denote by $X=H^{1}\left(\mathbb{T}_{L}\right)\times L^{2}\left(\mathbb{T}_{L}\right)$. Then, by the min-max principle again, we know that $\widetilde{\sigma}_1$ satisfies the following characterization:  \begin{align*}
	\widetilde{\sigma}_{1}=\max_{\left(\psi_{1},\psi_{2}\right)\in X}\min_{\substack{(\varphi_1, \varphi_2)\in X\backslash\left\{ 0\right\} \\
			(\varphi_1, \varphi_2)\perp\left(\psi_{1},\psi_{2}\right)
		}
	}\frac{\langle \vec{\mathcal{L}}_c(\varphi_1, \varphi_2),(\varphi_1, \varphi_2)\rangle}{\left\Vert (\varphi_1, \varphi_2)\right\Vert _{X}^{2}}.
\end{align*} 
Thus, by the Spectral Theorem, recalling the properties deduced in Property \ref{prop:linearspec} it immediately follows that for any function $\varphi\in H^1(\T_L)$ it holds: \[
\langle \varphi,Y_0\rangle=0 \quad \implies \quad \langle \mathcal{L}_c\varphi,\varphi\rangle\geq 0.
\]
Hence, by choosing $\psi_{1}=Y_{0}$ and $\psi_{2}=0$, by using the explicit form of $\mathcal{Q}_c$ in \eqref{eq:MQ} together with the latter inequality, we infer that 
\begin{align*}
	\widetilde{\sigma}_{1}\geq\min_{\substack{(\varphi_1, \varphi_2)\in X\backslash\left\{ 0\right\} \\
			(\varphi_1, \varphi_2)\perp\left(Y_{0},0\right)
		}
	}\frac{\langle\vec{\mathcal{L}}_c(\varphi_1, \varphi_2),(\varphi_1, \varphi_2)\rangle}{\left\Vert (\varphi_1, \varphi_2)\right\Vert _{X}^{2}}\geq0.
\end{align*}
Therefore, recalling that $\big\langle \vec{\phi}_{c,x},Y_{0}\big\rangle =0$ and $\vec{\phi}_{c,x}\in\ker\big(\vec{\mathcal{L}}_{c}\big)$, we conclude $\widetilde{\sigma}_{1}=0$. Finally, we follow a similar approach to obtain the needed information about $\widetilde{\sigma}_{2}$. In fact, by using the min-max principle once again, we can write $\widetilde{\sigma}_2$ as
\begin{align*}
	\widetilde{\sigma}_{2}=\max_{\substack{\left(\psi_{1},\psi_{2}\right)\in X\\
			\left(\psi_{3,}\psi_{4}\right)\in X
		}
	}\min_{\substack{(\varphi_1, \varphi_2)\in X\backslash\left\{ 0\right\} \\
			(\varphi_1, \varphi_2)\perp\left(\psi_{1},\psi_{2}\right)\\
			(\varphi_1, \varphi_2)\perp\left(\psi_{3},\psi_{4}\right)
		}
	}\frac{\langle \vec{\mathcal{L}}_{c}(\varphi_1, \varphi_2),(\varphi_1, \varphi_2)\rangle}{\left\Vert (\varphi_1, \varphi_2)\right\Vert _{X}^{2}}.
\end{align*}
Thus, in the same fashion  as before, by taking $\left(\psi_{1},\psi_{2}\right)=\left(Y_{0},0\right)$ as well as $\left(\psi_{3},\psi_{4}\right)=\left(\phi_{c}',0\right)$, as an application of the Spectral Theorem and Property \ref{prop:linearspec} we infer that
\begin{align*}
\widetilde{\sigma}_{2}=\min_{\substack{(\varphi_1, \varphi_2)\in X\backslash\left\{ 0\right\} \\
			\varphi_1\perp Y_{0}, \, \varphi_1\perp\phi_{c}'
		}
	}\frac{\langle \vec{\mathcal{L}_{c}}(\varphi_1, \varphi_2),(\varphi_1, \varphi_2)\rangle}{\left\Vert (\varphi_1, \varphi_2)\right\Vert _{X}^{2}}>0.
\end{align*}
More precisely, we have used the the explicit form of $\mathcal{Q}_c$ in \eqref{eq:MQ}, as well as the fact that $Q_{c}(\varphi_1, \varphi_1)\geq\sigma_{2}\left\Vert \varphi_1\right\Vert _{L^{2}\left(\mathbb{T}_{L}\right)}$
for any function $\varphi_1\in H^1(\T_L)$ satisfying $\varphi_1\perp Y_{0}$ and $ \varphi_1\perp\phi_{c}'$, where $\sigma_{2}$ is the third eigenvalue of $\mathcal{L}_{c}$ (which is positive by Proposition \ref{prop:linearspec}). Summarizing, we have proven that $\lambda_{0}<0$, $\lambda_{1}=0$ were both eigenvalues are simple, and $\lambda_{2}>0$, which concludes the proof.
\end{proof}
To finish this section, we consider the spectrum of $\vec{\mathcal{L}}_{c}$ for the
standing solution $S(x):=\phi_{0}$ restricted onto odd functional
spaces.
\begin{lem}\label{lem:oddspec}Under the assumptions of Theorem \ref{thm_monotonicity} the operator $\vec{\mathcal{L}}$ with $c=0$, that is $\vec{\mathcal{L}}_{0}$, defined in $L_{\text{odd}}^{2}\left(\mathbb{T}_{L}\right)\times L_{\text{odd}}^{2}\left(\mathbb{T}_{L}\right)$
 with domain $H_{\text{odd}}^{2}\left(\mathbb{T}_{L}\right)\times H_{\text{odd}}^{1}\left(\mathbb{T}_{L}\right)$,
 defines as a bounded self-adjoint operator with no negative eigenvalues and $\ker(\vec{\mathcal{L}})=\{(0,0)\}$
	in $L_{\text{odd}}^{2}\left(\mathbb{T}_{L}\right)\times L_{\text{odd}}^{2}\left(\mathbb{T}_{L}\right)$.
	Moreover, the rest of the spectrum is discrete and bounded away from
	zero.
\end{lem}

\begin{proof}
	First of all,  notice that, since  $Y_0$  is associated to the smallest eigenvalue of $\mathcal{L}_c$, it immediately follows that $Y_0$ is an even function regarded as a function in the whole line $\R$. On the other hand, we already know that $S'=\phi_{0}'$ is even in $\R$. Thus, none of these two eigenfunctions can belong to $H^2_{\mathrm{odd}}$. Therefore, gathering this information with Proposition \ref{prop:linearspec} we infer that the spectra of $\mathcal{L}_s:=\mathcal{L}_0$ which is defined in $L^2_{\mathrm{odd}}(\T_L)$ with domain $H^2_{\mathrm{odd}}(\T_L)$ is strictly positive. Hence, by the spectral theorem we deduce that, for all $\varphi_1\in H_{\text{odd}}^{1}\left(\mathbb{T}_{L}\right)$, it holds
\begin{align}\label{coer_c_not}
\langle\mathcal{L}_s\varphi_1, \varphi_1\rangle\geq\sigma\left\Vert \varphi_1\right\Vert _{L^{2}\left(\mathbb{T}_{L}\right)},
\end{align}
for some $\sigma\geq \sigma_2>0$. Then the same arguments as in the proof of Lemma \ref{lem:matrixspe} above, using the min-max principle gives us the desired results for the matrix operator
$\vec{\mathcal{L}}$.
\end{proof}

\section{Orbital Stability of standing waves in the odd energy space}\label{stab_standing}

From now on, and for the rest of this section, we shall always assume that $c=0$ and $L>0$ arbitrary but satisfies the hypothesis of Theorem \ref{thm_monotonicity}. Our goal here is to use the spectral analysis carried out in the previous section to prove the orbital stability of standing solutions under the additional hypothesis of global (in time) spatial oddness. In fact, one important advantage in this case is given by the preservation of the spatial-oddness by the periodic flow of the $\phi^{4n}$-equation. That is, if the initial data is $(\mathrm{odd},\mathrm{odd})$, then so is the solution associated to it for all times in the maximal existence interval. Then, recalling that  the traveling wave solution $\phi_c(x)$ constructed in Section \ref{existence} is odd, we obtain that if the initial perturbation $\vec\varepsilon_0=(\varepsilon_{0,1},\varepsilon_{0,2})=(\mathrm{odd},\mathrm{odd})$ and $c=0$, then so is the solution associated to \[
(\phi_{0,1},\phi_{0,2})=(S,0)+(\varepsilon_{0,1},\varepsilon_{0,2}).
\]
Thus, it is natural to study the time evolution of an initial odd perturbation of $(S,0)$ in terms of the evolution of its perturbation $\vec{\varepsilon}(t)$. In other words, for all times we shall write the solution as $\vec{\phi}(t,x)=(S(x),0)+\vec{\varepsilon}(t,x)$.  Additionally, by using equation \eqref{phif_2} and Taylor expansion we deduce that $\vec{\varepsilon}(t,x)$ satisfy the first-order system
\begin{align}\label{eps_system}
\begin{cases}
\partial_t\varepsilon_1=\varepsilon_2,
\\ \partial_t\varepsilon_2=-\mathcal{L}_{\mathrm{s}}\varepsilon_1+\mathcal{O}(\varepsilon_1^2),
\end{cases}
\end{align}
where $\mathcal{L}_s$ is the linearized operator around $S$, which is given by: \[
\mathcal{L}_s=-\partial_x^2+V_n''(S).
\]
Now, on the one-hand, from the spectral analysis developed in the previous section, we know that there is only one negative eigenvalue associated with the operator $\mathcal{L}_s$. Even more, we recall that both, $Y_0$ and $S'(x)$, are even functions (regarded as functions defined in the whole line $\R$). Moreover, it is not too difficult to see that periodic odd and even functions (where the parity is regarded as functions defined in the whole line $\R$) belonging to $H^1(\T_L)$ are orthogonal in the corresponding $H^1(\T_L)$-inner product. Thus, gathering all the analysis above we are in position to establish the following lemma.

\begin{lem}\label{coercivity}
Under the assumptions of Theorem  \ref{thm_monotonicity} the following holds: There exists $\gamma>0$ such that for any odd function $\upsilon\in H^1_{\mathrm{odd}}(\T_L)$ we have \[
\langle \mathcal{L}_{s}\upsilon,\upsilon\rangle \geq \gamma\Vert \upsilon\Vert_{H^1_{\mathrm{odd}}}^2.
\]
\end{lem}

\begin{proof}
In fact, first of all notice that we already know that the desired inequality holds if we change the $H^1$ norm for the $L^2$ in the right-hand side (see inequality \eqref{coer_c_not}). Now, we shall prove that by lowering the constant $\sigma$ we can \emph{improve} the latter inequality to put the $H^1$-norm in the right-hand side. In fact, by using the definition of $\mathcal{L}_s$ above, it immeditely follows that \begin{align}\label{QH}
	\left\Vert \upsilon_{x}\right\Vert _{L^{2}\left(\mathbb{T}_{L}\right)}^{2}= \langle \mathcal{L}_s\upsilon,\upsilon\rangle-\int_{\mathbb{T}_{L}}V''_n\big(S(x)\big)\upsilon^2\,dx.
\end{align}	
On the other hand, from the coercivity property given by Proposition \eqref{coer_c_not} it follows that, for any pair of positive numbers $\delta$ and $\eta$, and any odd periodic function $\upsilon \in H^1_{\mathrm{odd}}(\T_L)$ we have
\begin{align*}
\delta\left\Vert \upsilon_{x}\right\Vert _{L^{2}\left(\mathbb{T}_{L}\right)}^{2}+\eta\left\Vert \upsilon\right\Vert _{L^{2}\left(\mathbb{T}_{L}\right)}^{2}&\lesssim \delta\left\Vert \upsilon_{x}\right\Vert _{L^{2}\left(\mathbb{T}_{L}\right)}^{2}+\tfrac{\eta}{\sigma}\langle\mathcal{L}_s\upsilon, \upsilon\rangle
\\ &\lesssim\left(\delta+\tfrac{\eta}{\sigma}\right)\langle\mathcal{L}_s\upsilon, \upsilon\rangle+\delta C_{n}\Vert \upsilon\Vert^2 _{L^{2}\left(\mathbb{T}_{L}\right)},
\end{align*}
where in latter inequality we have used identity \eqref{QH} and $C_{n}:=\sup_{x\in\mathbb{T}_{L}}\vert V_{n}''(S(x))\vert$. Then, performing direct computations, reorganizing the latter inequality we conclude that
\begin{align*}
	\delta\Vert \upsilon_{x}\Vert _{L^{2}\left(\mathbb{T}_{L}\right)}^{2}+\left(\eta-\delta C_{n}\right)\left\Vert \upsilon\right\Vert _{L^{2}\left(\mathbb{T}_{L}\right)}^{2}\lesssim\left(\delta+\frac{\eta}{\sigma}\right)\langle\mathcal{L}_s\upsilon, \upsilon\rangle.
\end{align*}
	Choosing appropriate $\delta$ small enough and $\eta$ sufficiently large so that $\eta-\delta C_{n}>0$, it follows that there exists $\gamma>0$ which depends $C_{n}$
	such that
	\begin{equation}
	\langle\mathcal{L}_s\upsilon, \upsilon\rangle\geq\gamma\left\Vert \upsilon\right\Vert _{H^{1}\left(\mathbb{T}_{L}\right)}^{2}.\label{eq:energycoer}
	\end{equation}
The proof is complete.
\end{proof}
As an important application of the latter lemma we are able to improve the coercivity of the linearized Hamiltonian in \eqref{eq:mL} in the odd energy space to put the $X$-norm in the right-hand side (exactly as in the previous lemma). In fact, let us start by recalling that, in this case ($c=0$), the matrix quadratic form is given by:
\begin{align*}
\big\langle\vec{\mathcal{L}}_0(\upsilon_1, \upsilon_2),(\upsilon_1, \upsilon_2)\big\rangle =\int_{\mathbb{T}_{L}}\big(\upsilon_{1,x}^{2}+V_{n}''(S)\upsilon_1^{2}+\upsilon_2^{2}\big)dx.
\end{align*}
Applying the coercivity \eqref{eq:energycoer} to the first part of the matrix quadratic form it immediately follows that
\begin{align*}
\big\langle\vec{\mathcal{L}}_0(\upsilon_1, \upsilon_2),(\upsilon_1, \upsilon_2)\big\rangle\geq\gamma\left\Vert \upsilon_1\right\Vert _{H^{1}\left(\mathbb{T}_{L}\right)}^{2}+\left\Vert \upsilon_2\right\Vert _{L^{2}\left(\mathbb{T}_{L}\right)}^{2},
\end{align*}
which in particular implies that, for any odd $(\upsilon_1,\upsilon_2)\in H^1_{\mathrm{odd}}(\T_L)\times L^2_{\mathrm{odd}}(\T_L)$, one has the desired coercivity
\begin{align}\label{matrix_coerc}
\big\langle\vec{\mathcal{L}}_0(\upsilon_1,\upsilon_2),(\upsilon_1,\upsilon_2)\big\rangle\geq\widetilde{\gamma}\left\Vert (\upsilon_1, \upsilon_2)\right\Vert _{X_\text{odd}}^{2}.
\end{align} 	
With the information above we are in position to establish our orbital stability result.
\begin{thm}
	\label{thm:staodd} Consider $n\in\N$ arbitrary but fixed and let $L>\sqrt{2}\delta_n$ so that the hypothesis of Theorem \ref{thm_monotonicity} holds with $c=0$. The periodic standing wave solution $\vec{S}(x)=(S(x),0)$ is orbitally stable in the odd energy space $X_{\mathrm{odd}}$ under the periodic flow of the $\phi^{4n}$-equation.  More precisely, there exists $\delta>0$ small enough such that for any initial data
\begin{align*}
\vec{\varepsilon}_{0}=\big(\varepsilon_{0,1},\varepsilon_{0,2}\big)\in H_{\mathrm{odd}}^{1}\left(\mathbb{T}_{L}\right)\times L_{\mathrm{odd}}^{2}\left(\mathbb{T}_{L}\right),
\end{align*}
satisfying $\left\Vert \big(\varepsilon_{0,1},\varepsilon_{0,2}\big)\right\Vert_{X_{\mathrm{odd}}}\leq\delta$ the following holds: There exists a constant $C>0$ such that the solution to equation \eqref{eps_system} associated to $\vec{\varepsilon}_0$ satisfies:
\begin{align*}
\hbox{for all }\,t\in\R, \quad 	\big\Vert \big(\varepsilon_{1}(t),\varepsilon_{2}(t)\big)\big\Vert_{X_{\text{odd}}}\leq C\delta.
\end{align*}
\end{thm}
\begin{proof}
In fact, by  the smallness assumption of the initial data $(\varepsilon_{0,1},\varepsilon_{0,2})$, it follows that \[
\big\vert\mathcal{E}\big(\vec{\phi}_{0}\big)-\mathcal{E}\big(\vec{S}\big)\big\vert \lesssim\delta,
\]
where $\mathcal{E}$ is the conserved energy functional defined in \eqref{energy} with $\lambda_n=1$. Then, for any $t\in\mathbb{R}$, explicitly computing the differences of the energies between $(\phi_1,\phi_2)$ and $(S,0)$, we get
\begin{align}
\mathcal{E}\big(\vec{\phi}\big)-\mathcal{E}\big(\vec{S}\big)=\frac{1}{2}\int_{\T_L} \Big(\varepsilon_{1,x}^{2}+2\varepsilon_{1,x}S'+\varepsilon_{2,x}^{2}+2V_{n}\big(S+\varepsilon_{1}\big)-2V_n(S)\Big)dx\label{diffE}
\end{align}
For the last two terms inside the integral in the latter identity, by Taylor expansion we infer
\begin{align*}
V_{n}\big(S+\varepsilon_{1}\big)=V_{n}(S)+V_{n}'(S)\varepsilon_{1}+\frac{1}{2}V_{n}''(S)\varepsilon_{1}^{2}+\mathcal{O}\big(\varepsilon_{1}^{3}\big).
\end{align*}
Hence, integrating by parts applied to the second integrand of \eqref{diffE} and  using the equation solved by $S$, that is, replacing $S''=-V'_n(S)$, we can write
\begin{align*}
\big\vert\mathcal{E}\big(\vec{\phi}\big)-\mathcal{E}\big(\vec{S}\big)\big\vert & =\frac{1}{2}\int_{\T_L}\big(\varepsilon_{1,x}^{2}+\varepsilon_{2}^{2}+V_{n}''\left(S\right)\varepsilon_{1}^{2}+\mathcal{O}\big(\varepsilon_{1}^{2}\big)\big)dx 
\\ & \gtrsim \gamma_{n}\left\Vert \left(\varepsilon_{1},\varepsilon_{2}\right)\right\Vert _{X_\text{odd}}^{2}+\left\Vert \varepsilon_{1}\right\Vert _{H^{1}\left(\mathbb{T}_{L}\right)}^{3} \gtrsim\big\Vert (\varepsilon_{1},\varepsilon_{2})\big\Vert _{X_{\text{odd}}}^{2}.\nonumber 
\end{align*}
where in the first inequality above we have used the coercivity property \eqref{matrix_coerc} and Sobolev embedding. Therefore, due to the energy conservation we conclude that, for any $t\in\mathbb{R}$ the following holds:
\begin{align*}
	\big\Vert (\varepsilon_{1}(t),\varepsilon_{2}(t))\big\Vert_{X_{\text{odd}}}^{2} & \lesssim\big\vert\mathcal{E}\big(\vec{\phi}(t)\big)-\mathcal{E}\big(\vec{S}\big)\big\vert \lesssim\big\vert\mathcal{E}\big(\vec{\phi}_{0}\big)-\mathcal{E}\big(\vec{S}\big)\big\vert\lesssim\delta.
\end{align*} 
The proof is complete.
\end{proof}

\section{Orbital instability of traveling waves in the whole space\label{sec:Instability}}

In this section, we gather the spectral information of the linearized Hamiltonian given in Lemma \ref{lem:matrixspe} and the general result of Grillakis-Shatah-Strauss in  \cite{GSS} to establish orbital instability of traveling wave solutions $\vec{\phi}_{c}\left(x-ct\right)$ under general perturbations in the energy space. It is worthwhile to notice that in order to apply the main result in \cite{GSS} we need to analyze the sign of $\tfrac{d}{dc}\Vert \phi_{c,x}\Vert_{L^2}^2$. However, the fact that no explicit formula for the solution $\vec{\phi}_{c}$ exists presents a hard obstacle to overcome. In order to surpass this difficulty, the monotonicity of the period $L$ with respect to $\beta$ shall play a key role in our analysis. Finally, we remark that, without loss of generality, from now on we shall assume always that $c>0$.
\begin{thm}
	\label{thm:instab}
	Consider $n\in\N$ and let $L>0$ be arbitrary but fixed. Under the hypothesis of Theorem \ref{thm_monotonicity}, the traveling wave $\vec{\phi}_{c}\left(x-ct\right)$ constructed in Section \ref{existence} is orbitally unstable in the energy space $H^{1}\left(\mathbb{T}_{L}\right)\times L^{2}\left(\mathbb{T}_{L}\right)$ under the periodic flow of the $\phi^{4n}$-equation.
\end{thm}

\begin{proof}
We recall that, by the standard Grillakis-Shatah-Strauss theory (see the main result in \cite{GSS}), we know that, once the existence of the smooth curve of traveling waves solutions and the main spectral information
 of the linearized Hamiltonian around $\vec{\phi}_{c}$ are established, the (in)stablity problem is reduced to study the convexity/concavity of the scalar function \begin{align*}
d(c):=\mathcal{E}\big({\vec{\phi_{c}}}\big)+c\mathcal{P}\big(\vec{\phi_{c}}\big).
\end{align*}
In our current setting, the traveling wave $\vec{\phi}_{c}$ is orbitally stable if and only if $d(c)$ is strictly convex and unstable if and only if $d(c)$ is strictly concave. In other words, the orbital instability is equivalent to show that $d''(c)>0$. Moreover, recalling that $\vec{\phi_{c}}$ is a critical point of the action functional $\mathcal{E}+c\mathcal{P}$, we deduce that
\[
d'(c)=-c\int_{\mathbb{T}_{L}}\phi_{c,x}^{2}\,dx.
\]
Thus, in order to analyze the concavity/convexity of $d(c)$ we differentiate the latter identity with respect to $c$, from where we get 
\begin{align}\label{asdasd}
d''(c)=\frac{d}{dc}\left(-c\int_{\mathbb{T}_{L}}\left(\phi_{c,x}\right)^{2}\,dx\right)=-\int_{\mathbb{T}_{L}}\left(\phi_{c,x}\right)^{2}\,dx -c\frac{d}{dc}\int_{\mathbb{T}_{L}}\left(\phi_{c,x}\right)^{2}\,dx.
\end{align}
Now, for the sake of simplicity, from now on we shall denote by $\eta_c$ the derivative of the solution with respect to the velocity $c$, that is, $\eta_{c}:=\frac{d}{dc}\phi_{c}$. We claim that the right-hand side of \eqref{asdasd} is strictly negative. Then, it suffices to analyze
\begin{align}\label{asdasdasd}
-c\frac{d}{dc}\int_{\mathbb{T}_{L}}\left(\phi_{c,x}\right)^{2}\,dx & =-2c\int_{\mathbb{T}_{L}}\eta_{c,x}\phi_{c,x}\,dx.
\end{align}
We intend to prove that the right-hand side of the latter equation is negative. In fact, first of all let us recall that once $L$ and $c$ are fixed, the traveling wave solution satisfies the Hamiltonian equation with energy $\beta$ (see \eqref{hamilt}, \eqref{gamma}): \begin{align}\label{ham_energy_beta}
\dfrac{1}{2}\phi_x^2-\dfrac{1}{\omega}\big(V_n(\phi)-V_n(0)\big)=\beta.
\end{align}
Hence, by differentiating the latter equation with respect to the speed $c$ we obtain
\begin{align*}
	\phi_{c}'\eta_{c}'-\frac{2c}{\omega^2}\left(V_{n}\left(\phi_{c}\right)-V_{n}\left(0\right)\right)-\frac{1}{\omega}V_{n}'\left(\phi_{c}\right)\eta_{c}=\frac{d}{dc}\beta.
\end{align*}
Recalling that the traveling wave solution satisfies $\frac{1}{\omega}V_{n}'\left(\phi_{c}\right)=\phi_{c}''$, we can rewrite the latter equation in the more convenient form as 
\begin{align*}
	\phi_{c}'\eta_{c}'-\frac{2c}{\omega^2}\big(V_{n}\left(\phi_{c}\right)-V_{n}(0)\big)-\phi_{c}''\eta_{c}=\frac{d}{dc}\beta.
\end{align*}
Then, by integrating the latter equation over $\mathbb{T}_{L}$ and performing some integration  by parts it follows 
\begin{align*}
	2\int_{\mathbb{T}_{L}}\phi_{c}'\eta'_{c}\,dx=L\frac{d\beta}{dc}+\frac{2c}{\omega^2}\int_{\mathbb{T}_{L}}\big(V_{n}(\phi_{c})-V_{n}(0)\big)dx.\end{align*}
Plugging identity \eqref{ham_energy_beta} to the last term on the right-hand side above, we deduce
\begin{align}\label{ham_2}
2\int_{\mathbb{T}_{L}}\phi_{c}'\eta'_{c}\,dx=\left(\frac{d}{dc}\beta-\frac{2c}{\omega}\beta\right)L+\frac{c}{\omega}\int_{\mathbb{T}_{L}}\phi_{c,x}^{2}\,dx.
\end{align}
Therefore, it suffices to analyze the sign of the inner parenthesis $\frac{d}{dc}\beta-\frac{2c}{\omega}\beta$. To achieve this, recalling the formula for the period in \eqref{eq:peri}, we define 
\begin{align}
	L  = \sqrt{2\omega}\int_{x_0}^{x_1}\tfrac{dx}{\sqrt{\omega\beta-\big(V_n(x)-V_n(0)\big)}} =:\sqrt{2\omega}\widetilde{L}.\label{eq:periodfor1}
	\end{align}
For the sake of simplicity, from now on we shall denote by $\widetilde{\beta}:=\omega\beta$. Hence, as a direct application of Theorem \ref{thm_monotonicity} we infer that $\frac{d}{d\widetilde{\beta}}\widetilde{L}>0$. Differentiating formula \eqref{eq:periodfor1} with respect $c$, recalling that the period $L$ is fixed, we obtain
\begin{align}\label{dldc}
0=-\frac{\sqrt{2}c}{\sqrt{\omega}}\widetilde{L}+\sqrt{2\omega}\frac{d\widetilde{L}}{d\widetilde{\beta}}\frac{d\widetilde{\beta}}{dc}.\end{align}
On the other hand, it easily follows from the definition of $\widetilde{\beta}$ that $\frac{d\widetilde{\beta}}{dc}=-2c\beta+\omega\frac{d\beta}{dc}$. Thus, plugging this relation into \eqref{dldc} and recalling that $d\widetilde{L}/d\widetilde{\beta}>0$ we infer
\[
\omega\sqrt{2\omega}\frac{d\widetilde{L}}{d\widetilde{\beta}}\left(\frac{d\beta}{dc}-\frac{2c}{\omega^2}\beta\right)=\frac{\sqrt{2}c}{\sqrt{\omega}}\widetilde{L} \quad \implies \quad \dfrac{d\beta}{dc}-\dfrac{2c}{\omega^2}\beta>0.
\]
Finally, by plugging the latter inequality into \eqref{ham_2} and recalling identity \eqref{asdasdasd} we conclude\footnote{Notice that we arrive at the same conclusion if $c<0$.}
\begin{align}
	\frac{d}{dc}\left(-c\int_{\mathbb{T}_{L}}\left(\phi_{c,x}\right)^{2}\,dx\right) \leq-\frac{1}{\omega}\int_{\mathbb{T}_{L}}\left(\phi_{c,x}\right)^{2}\,dx.\nonumber 
	\end{align}
Therefore, $d''\left(c\right)<0$,  and hence, by using the main result in \cite{GSS} we conclude that $\vec{\phi}_{c}\left(x-ct\right)$ is orbitally unstable in the energy space.
\end{proof}

\section{Extension of the main result in \cite{deNa}}\label{ext_nata}

\subsection{The Model} As mentioned in the introduction, as a  by-product of our current analysis we are able to extend the main result in \cite{deNa} to general $n\in\N$. In order to avoid misunderstandings with our previous equation, from now on we shall denote the unknown by $\varphi(t,x)$. With this in mind, in the sequel we shall consider the following type of generalization of the $\phi^{4}$-equation, that we shall call $\varphi^{2n+2}$-equation: 
\begin{align}\label{phi_w}
\partial_t^2\varphi-\partial_x^2\varphi-\varphi+\varphi^{2n+1}=0.
\end{align}
Before recalling the main results in \cite{deNa}, let us start by introducing some notations. To be consistent with our analysis above, we rewrite equation \eqref{phi_w} as
\begin{align}
\partial_{t}^{2}\varphi-\partial_{x}^{2}\varphi+\widetilde{V}_{n}'(\varphi)=0,\label{eq:phi_w_v}
\end{align}
where in this setting the potential is given by
\begin{align}\label{pot2}
\widetilde{V}_{n}(x):=-\frac{x^{2}}{2}+\frac{x^{2n+2}}{2\left(n+1\right)}.
\end{align}
We point out that, in sharp contrast with model \eqref{phif}, the potential associated to in \eqref{pot2} is not getting additional different minima as $n$ increases, and hence, no more soliton sectors. Instead, potential \eqref{pot2} has always (for all $n\in\N$) exactly $3$ real roots, which are located at $0,\,(n+1)^{1/2n},\,-(n+1)^{1/2n}$. Notice that $0$ has always multiplicity $2$, while both $\pm (n+1)^{1/2n}$  have  multiplicity $\frac{n}{2}$ if $n$ is even and multiplicity $n$ otherwise.

\medskip

%Trivialy, $\tilde{V}_{n}(0)=0$. 
On the other hand, we can write equation \eqref{eq:phi_w_v} as a first order system for $\vec{\varphi}=(\varphi_1,\varphi_2)$ as \begin{align}\label{phif_w_v}
\begin{cases}
\partial_t\varphi_1=\varphi_2,
\\ \partial_t\varphi_2=\partial_x^2\varphi_1-\tilde{V}_n'(\varphi_1).
\end{cases}
\end{align}
As for model \eqref{phif} above, the Hamiltonian structure of the system gives us the energy conservation of  \eqref{phif_w_v}, that is, the following functional is conserved along the flow:
\begin{align}\label{energy2} 
\widetilde{\mathcal{E}}(\vec{\varphi}(t))&:=\dfrac{1}{2}\int_0^L \big(\varphi_2^2+\varphi_{1,x}^2+2\widetilde{V}_n(\varphi_1)\big)(t,x)dx=\widetilde{\mathcal{E}}(\vec{\varphi}_0).
\end{align}
We also have the conservation of momentum which is given by:
\begin{align}\label{momentum2}
\widetilde{\mathcal{P}}(\vec{\varphi}(t)):=\int_0^L \varphi_2(t,x)\varphi_{1,x}(t,x)dx=\widetilde{\mathcal{P}}(\vec{\varphi}_0).
\end{align}
%We point out that from these conservation laws it follows that $H^1(\T_L)\times L^2(\T_L)$ defines the natural energy space associated to system \eqref{phi_w}.

Regarding the results in \cite{deNa}, in the case of traveling wave solutions orbiting around $(0,0)$, the authors in \cite{deNa} were able to prove the orbital instability in the whole energy space for $n=1,2$. However, the proof of the sign of $d''(c)$ relies in some  numerical computations and is argued by the plot of a ``hidden'' function (they have not provided the function they are plotting to justify this sign). This is an important remark, since being able to compute the sign of $d''(c)$ is usually a very challenging part of the analysis, and is the only reason why the authors in \cite{deNa} are not able to extend their result for $n$ larger than $2$. In this section, we intend to extend their orbital instability result for all values of $n\in\N$, that is, for all $\varphi^{2n+2}$-equations.

\medskip

We point out that, at the date of this publication, the proof of the results in \cite{deNa} contain typos, some of them problematic, but in such a way that they ``cancel each other'' so that the authors end up with the correct conclusions. Thus, in order to extend their result we need to start by fixing some of these typos since they would provoke different conclusions in our analysis.

\subsection{Extension of the main result in \cite{deNa}}

We start by recalling the Hamiltonian system satisfied by traveling wave solutions:
\begin{align}\label{system1}
\begin{cases}\dot{u}=v,
\\ \dot{v}=\tfrac{1}{\omega}\widetilde{V}_{n}'(u),
\end{cases}
\end{align}
where $\omega:=1-c^2$.
By direct computations $\widetilde{V}'_n(u)=u(u^n-1)(u^n+1)$, so the system above has three critical points: $(u,v)=(0,0)$ which is a stable center, and $(u,v)=(\pm1,0)$ which are both saddle points. 
%{\color{red} can you show/give a reason why there is no critical point in %between 0 and 1 (this is important)? Compute them all.}
We also recall that the Hamiltonian assoicated the system \ref{system1} is
\begin{align}
\widetilde{\mathcal{{H}}}\left(u,v\right):=\frac{1}{2}v^{2}-\frac{1}{\omega}\widetilde{V}_n(u)=\frac{1}{2}v^{2}+\frac{1}{\omega}\left(\frac{u^{2}}{2}-\frac{u^{2n+2}}{2\left(n+1\right)}\right)\label{hamilt1}
\end{align}

Therefore, by the standard ODE theory for Hamiltonian equations (see for example \cite{ChiL}), we know that all periodic solutions of \eqref{system1} orbiting around $(0,0)$ corresponds to regular level sets of $\widetilde{\mathcal{H}}$, with energy $\beta\in(0,\widetilde{E}_\star)$, where the maximal energy level $\widetilde{E}_\star=\frac{n}{2\omega (n+1)}$. Finally, we recall the period $L$ can be express as
\begin{equation}
L=\sqrt{2}\int_{x_0}^{x_1}\tfrac{dx}{\sqrt{\beta+\frac{1}{\omega}\widetilde{V}_n(x)}}.\label{eq:peri1}
\end{equation}

The main point now is to show the monotonicity of the period $L$ with respect
to the level of the energy $\beta$.

\medskip

\textbf{Monotonicity of the period map:} In section $2$ in \cite{deNa}, more specifically just below identity $(2.5)$, the authors claim that the period map $L(\beta)$  is strictly decreasing in $\beta$. Consequently, they claim that $L(\beta)$ goes to $+\infty$ when $\beta$ goes to zero, and that $L(\beta)$ converges to a finite constant when $\beta$ goes to the maximal possible energy $\beta_{\mathrm{max}}$. We give three different reasons why this cannot hold. First, in the first line of the proof of Lemma $2.1$ in \cite{deNa} the authors define the function $f(h):=-h+h^{2k+1}$. However, this definition of $f(h)$ does not match the definitions of the result the authors are refering to (see the definition of the Hamiltonian at the beginning of Section $2$ in \cite{deNa}). Specifically, $f(h)$ is missing a minus sign. Secondly, it is a well-known fact that $\varphi^4$ and $\varphi^6$ have explicit odd Kink solutions. Moreover, all of these odd Kink solutions belong to the separatrix curve, have finite energy and infinite period. This contradicts the fact that $L(\beta)'<0$ since it is mandatory for $L(\beta)$ to goes to $+\infty$ when $\beta\to\beta_{\mathrm{max}}$. Finally, by taking advantage of the fact that all the involved functions (including the period) are explicit in the case $n=1$, the second author of the present work has explicitly proved in \cite{Pa} that \[L(\beta)\to+\infty \ \hbox{ when  }\ \beta\to\beta_{\mathrm{max}} \quad \hbox{ and } \quad L(\beta)\to 2\pi\sqrt{\omega} \ \hbox{ when } \ \beta\to0,
\]
which also contradicts Lemma $2.1$ in \cite{deNa}. Summarizing, we have the following lemma.

\begin{lem}\label{lem:monotonicity2}
Given the formula \eqref{eq:peri1}, one has $\frac{dL}{d\beta}>0$, for all $\beta\in(0,\widetilde{E}_\star)$.
\end{lem}
\begin{proof}
 In the same fashion as before, by using the main result in \cite{Chi}, it suffices to show the convexity of $-\widetilde{V}_n(x)/(\widetilde{V}_n'(x))^2$
for all $x\in(-1,1)$.
In fact, by using the explicit form of $\widetilde{V}_n$, and after performing some direct computations we obtain
\[
-\dfrac{d^2}{dx^2}\dfrac{\widetilde{V}_{n}(x)}{(\widetilde{V}_{n}'(x))^2}=\frac{nx^{2n-2}\left(4n^{2}-1+2(1+n+4n^{2})x^{2n}-(1+2n)x^{4n}\right)}{(1+n)(-1+x^{2n})^{4}}.
\]
Since $x^{2n}\leq 1$ for $\vert x\vert <1$, and $n\geq1$, it immediately follows:
\[
4n^{2}-1+2(1+n+4n^{2})x^{2n}-(1+2n)x^{4n}\geq 3n^2+(1+4n^{2n})x^{2n}>0,\] 
which concludes the proof.
\end{proof}

%Now, by direct computations it is easy to prove that $\tfrac{d}{d\beta}L>0$ %{\color{red}write the whole proof of the monotonicity.}

%.
Then, by using the monotonicity of the period map, we deduce again the existence of a limit as $\beta\to0^+$. We shall recycle the notation and call this limit \[
\lim_{\beta\to 0^+}L(\beta)=:\sqrt{\omega}\widetilde{\delta_n}.
\]

\textbf{Spectral analysis:}
To study (in)stability of traveling waves, in \cite{deNa}, the authors analyzed the linearized operator around the traveling wave which  is given by
\[
\mathcal{L}_{c}y:=-\omega y''+\left(-y+(2n+1)\varphi_{c}^{2n}y\right).
\]
In Section 3 of \cite{deNa}, the authors applied Theorem \ref{thm:Neves2} (of the present work) to study the spectral properties of the scalar linearzied operator.  But, in Lemma 3.2 in \cite{deNa}, the authors obtained the relation $\theta=\frac{dL}{d\beta}$. Nevertheless, notice that after fixing the sign of $\tfrac{dL}{d\beta}$, this latter relation, together with the spectral analysis carried out in \cite{deNa}, would lead to different conclusions (so that it would not be possible to conclude the main theorems in \cite{deNa}). However, it turns out that this relation $\theta=\tfrac{dL}{d\beta}$ is not correct. The essential reason is that, in order to use the quantity $\theta$ defined in Theorem \ref{thm:Neves1}, one has to ensure that the Wronskian determinant is $1$ in the right order (notice that they have switched the order of the entries in the Wronskian to obtain an extra minus sign, see Theorem \ref{thm:Neves1} above or \cite{N} for further details). Thus, with this wrong relation $\theta=\frac{dL}{d\beta}$ and the opposite sign of $\frac{dL}{d\beta}$, the authors could still conclude the correct spectral properties of the linear scalar operator due to this ``cancellation'' of double minus signs (see Theorem \ref{thm:Neves2} above or \cite{N} to see the impact of the sign of $\theta$ in the spectral information). Summarizing, we have the following lemma.

\begin{lem} 
Under our current hypothesis, the following relation between holds: $\theta=-\frac{\partial L}{\partial\beta}$.
\end{lem}

We point out that after fixing these two typos, the proofs given in \cite{deNa} follow.

\medskip

Finally, with the monotonicity of the period and spectral properties, we are able to extend the main result in \cite{deNa}.
\begin{thm}
Let $n\in\N$ and consider $L>0$ arbitrary but fixed. For any speed $c\in(-1,1)$ such that $L>\lambda_n$, the traveling wave solution $\vec{\varphi}_{c}\left(x-ct\right)$ constructed in Section $2$ in \cite{deNa} is orbitally unstable in the energy space $H^{1}(\mathbb{T}_{L})\times L^{2}(\mathbb{T}_{L})$ under the periodic flow of the $\varphi^{2n+2}$ model.
\end{thm}

\begin{proof}
Again, without loss of generality we shall assume $c>0$. The proof follows in a similar fashion as the one in the previous section, and hence we shall only give its main points. First, recall that by the main result of Grillakis-Shatah-Strauss in \cite{GSS}, the (in)stablity problem is reduced to study the convexity/concavity of the scalar function:
	\begin{equation}
	d\left(c\right):=\widetilde{\mathcal{E}}\left(\varphi_{c}\right)+c\widetilde{\mathcal{P}}\left(\varphi_{c}\right).\label{eq:dscalar}
	\end{equation}
	Once again, in our current setting, the traveling wave $\varphi_{c}$ is orbitally
	stable if and only if $d\left(c\right)$ is convex. Then, in a similar fashion as before, computing the second derivative of $d(c)$ we deduce that it suffices to analyze
\begin{align}\label{asdasdasd1}
-c\frac{d}{dc}\int_{\mathbb{T}_{L}}\left(\varphi_{c,x}\right)^{2}\,dx & =-2c\int_{\mathbb{T}_{L}}\widetilde{\eta}_{c,x}\varphi_{c,x}\,dx,
\end{align}
where $\widetilde{\eta}_{c}:=\frac{d}{dc}\varphi_{c}$.  Again, we will prove that the right-hand side of the  equation above is negative. In fact, by differentiating (with respect to the speed $c$) the Hamiltonian equation with energy $\beta$ and performing the same manipulation as our proof of Theorem \ref{thm:instab}, we obtain
\begin{align}\label{ham_21}
2\int_{\mathbb{T}_{L}}\varphi_{c}'\tilde{\eta}'_{c}\,dx=\left(\frac{d}{dc}\beta-\frac{2c}{\omega}\beta\right)L+\frac{c}{\omega}\int_{\mathbb{T}_{L}}\varphi_{c,x}^{2}\,dx.
\end{align}
On the other hand, proceeding exactly as before we deduce $\frac{d}{dc}\beta-\frac{2c}{\omega}\beta>0$. Therefore, by plugging the latter inequality into \eqref{ham_21} and recalling identity \eqref{asdasdasd1} we conclude\footnote{Notice that we arrive to the same conclusion if $c<0$.} 
\begin{align}
\frac{d}{dc}\left(-c\int_{\mathbb{T}_{L}}\varphi_{c,x}^{2}\,dx\right) \leq-\frac{1}{\omega}\int_{\mathbb{T}_{L}}\varphi_{c,x}^{2}\,dx.\nonumber 
\end{align}
Thus, $d''(c)<0$,  and hence, by using the main result in \cite{GSS} we conclude that $\vec{\varphi}_{c}(x-ct)$ is orbitally unstable in the energy space.

%Thus, in a similar fashion as before, by differentiating with respect to $c$ the Hamiltonian equation associated to \eqref{phi_w} we obtain $-c\frac{d}{dc}\int_{\mathbb{T}_{L}}\left(\varphi_{c,x}\right)^{2}\,dx\leq0.$
%{\color{red}way too short proof. You can shorten the proof but not that much. }
%Therefore, $d''\left(c\right)<0$, what concludes the proof by applying the main result in \cite{GSS}.
\end{proof}

%{\color{red} Use medskip if you break a line.}


\begin{thebibliography}{99}

\smallskip


\bibitem{AlMuPa2} M. A. Alejo, C. Mu\~noz, J. M. Palacios, \emph{On the variational structure of breather solutions I: Sine-Gordon equation}, J. Math. Anal. Appl. 453 (2017), no. 2, 1111–1138.

\bibitem{AlMuPa} M. A. Alejo, C. Mu\~noz, J. M. Palacios, \emph{On the variational structure of breather solutions II: Periodic mKdV equation}. Electron. J. Differential Equations 2017, Paper No. 56, 26 pp.

\bibitem{AlMuPa3} M. A. Alejo, C. Mu\~noz, J. M. Palacios, \emph{On the asymptotic stability of the sine-Gordon kink in the energy space}, preprint arXiv:2003.09358. 

\bibitem{An} J. Angulo Pava, \emph{Nonlinear stability of periodic traveling wave solutions to the Schr\"odinger and the modified Korteweg-de Vries equations}. J. Differential Equations 235 (2007), no. 1, 1–30.

\bibitem{AnNa} J. Angulo Pava, F. Natali, \emph{Positivity properties of the Fourier transform and the stability of periodic travelling-wave solutions}, SIAM J. Math. Anal. 40 (2008), no. 3, 1123–1151.

\bibitem{AnNa2} J. Angulo, F.  Natali,\emph{(Non)linear instability of periodic traveling waves: Klein-Gordon and KdV type equations}. Adv. Nonlinear Anal. 3 (2014), no. 2, 95–123.

\bibitem{BiFr} P. Byrd, M. Friedman, \emph{Handbook of Elliptic Integrals for Engineers and Scientists}, second ed., Springer-Verlag,
New York, 1971.

\bibitem{ChiL} C. Chicone, \emph{Ordinary Differential Equations with Applications}, Springer, New York, (2006).

\bibitem{Chi} C. Chicone, \emph{The monotonicity of the period function for planar Hamiltonian vector fields}.
J. Differential Equations 69 (1987), no. 3, 310–321.

\bibitem{Cu} S. Cuccagna, \emph{On asymptotic stability in 3D of kinks for the $\phi^4$ model}, Trans. Amer. Math. Soc. 360 (2008), no. 5, 2581–2614.

\bibitem{deNa} G. de Loreno, F. Natali, \emph{Odd Periodic Waves for some Klein-Gordon Type Equations: Existence and Stability}, preprint arXiv:2006.01305

\bibitem{DeKa} B. Deconinck, T. Kapitula, \emph{The orbital stability of the cnoidal waves of the Korteweg-de Vries equation}. Phys. Lett. A 374 (2010), no. 39, 4018–4022.

\bibitem{DeMc} B. Deconinck, P. McGill, B. Segal, \emph{The stability spectrum for elliptic solutions to the sine-Gordon equation}. Phys. D 360 (2017), 17–35.

\bibitem{DeUp} B. Deconinck, J. Upsal, \emph{The orbital stability of elliptic solutions of the focusing nonlinear Schrödinger equation}, SIAM J. Math. Anal. 52 (2020), no. 1, 1–41.

\bibitem{DeNi} B. Deconinck, M. Nivala, \emph{The stability analysis of the periodic traveling wave solutions of the mKdV equation}, Stud. Appl. Math. 126 (2011), no. 1, 17–48.

\bibitem{De} J-M Delort, \emph{Existence globale et comportement asymptotique pour l'\'equation de Klein-Gordon quasi lin\'eaire \`a donn\'ees petites en dimension 1}, Ann. Sci. Ecole Norm. Sup. 34(4) (2001) pp. 1–61.

\bibitem{GSS} M. Grillakis, J. Shatah, W. Strauss, \emph{Stability theory of solitary waves in the presence of symmetry I}, J. Funct. Anal. 74 (1987), no. 1, 160–197.

\bibitem{GSS2} M. Grillakis, J. Shatah, W. Strauss, \emph{Stability theory of solitary waves in the presence of symmetry II}, J. Funct. Anal. 94 (1990), no. 2, 308–348.

\bibitem{HaNa} N. Hayashi, P. Naumkin, \emph{The initial value problem for the cubic nonlinear Klein-Gordon equation}, Z. Angew. Math. Phys. 59 (2008), no. 6, 1002–1028.

\bibitem{HPW} D. Henry, J.  Perez, W. Wreszinski, \emph{Stability theory for solitary-wave solutions of scalar field equations}, Comm. Math. Phys. 85 (1982), no. 3, 351–361

\bibitem{In} E. Ince, \emph{The periodic Lam\'e functions}, Proc. Roy. Soc. Edinburgh 60 (1940) 47–63.

\bibitem{JoMaMiPl} C. Jones, R. Marangell, P. Miller, R. Plaza, \emph{Spectral and modulational stability of periodic wavetrains for the nonlinear Klein-Gordon equation}, J. Differential Equations 257 (2014), no. 12, 4632–4703.

\bibitem{JoMaMiPl2} C. Jones, R. Marangell, P. Miller, R. Plaza, \emph{On the stability analysis of periodic sine-Gordon traveling waves}, Phys. D 251 (2013), 63–74.


\bibitem{JoMaMiPl3} C. Jones, R. Marangell, P. Miller, R. Plaza, \emph{On the spectral and modulational stability of periodic wavetrains for nonlinear Klein-Gordon equations}. Bull. Braz. Math. Soc. (N.S.) 47 (2016), no. 2, 417–429.


\bibitem{Kato} T. Kato, \emph{ Quasi-Linear Equations of Evolution with Applications to Partial Differential Equations}, Lecture Notes in Math., vol. 448, Springer, 1975, pp. 25–
70.

\bibitem{KeCu} P. G. Kevrekidis and J. Cuevas-Maraver, \emph{A Dynamical Perspective on the $\phi^4$ Model. Past, Present and Future}. Nonlinear Systems and Complexity Series. Springer 2019.

\bibitem{Kl} S. Klainerman,
\emph{Global existence of small amplitude solutions to nonlinear Klein-Gordon equations in four space-time dimensions}, Comm. Pure Appl. Math. 38 (1985), no. 5, 631–641.

\bibitem{Kl2} S. Klainerman,
\emph{Global existence for nonlinear wave equations}. Comm. Pure Appl. Math. 33 (1980), no. 1, 43–101.

\bibitem{KMM} M. Kowalczyk,  Y. Martel, C. Mu\~noz, 
\emph{Kink dynamics in the $\phi^4$ model: asymptotic stability for odd perturbations in the energy space}. J. Amer. Math. Soc. 30 (2017), no. 3, 769–798. 

\bibitem{KMM2} M. Kowalczyk, Y. Martel, C. Mu\~noz, \emph{Nonexistence of small, odd breathers for a class of nonlinear wave equations}, Lett. Math. Phys. 107 (2017), no. 5, 921–931.

\bibitem{KMMH} M. Kowalczyk, Y. Martel, C. Mu\~noz, H. Van Den Bosch, \emph{A sufficient condition for asymptotic stability of kinks in general (1+1)-scalar field models},  preprint arXiv:2008.01276

\bibitem{LiSo} H. Lindblad, A. Soffer, \emph{Scattering for the Klein-Gordon equation with quadratic and variable coefficient cubic nonlinearities}. Trans. Amer. Math. Soc. 367 (2015), no. 12, 8861–8909.


\bibitem{Lo} M. A. Lohe, \emph{Soliton structures in $P(\phi)_2$}, Physical Review D, 20 (1979), 3120-3130

\bibitem{MaWi} W. Magnus, S. Winkler, \emph{Hill’s Equation}, Tracts Pure Appl. Math., vol. 20, Wiley, New York, 1976.


\bibitem{MaPa} N. Manton, P. Sutcliffe, \emph{Topological solitons}, Cambridge Monographs on Mathematical
Physics, Cambridge University Press, Cambridge, 2004.

\bibitem{MP} C. Mu\~noz, J. M. Palacios, \emph{Nonlinear stability of 2-solitons of the sine-Gordon equation in the energy space}, Ann. Inst. H. Poincar\'e Anal. Non Linéaire 36 (2019), no. 4, 977–1034.


\bibitem{NaCa} F. Natali, E. Cardoso, \emph{Stability properties of periodic waves for the Klein-Gordon equation with quintic nonlinearity}, Appl. Math. Comput. 224 (2013), 581–592.

\bibitem{NaPa} F. Natali, A. Pastor Ferreira, \emph{Stability and instability of periodic standing wave solutions for some Klein-Gordon equations},
J. Math. Anal. Appl. 347 (2008), no. 2, 428–441.

\bibitem{N} A. Neves,\emph{ Floquet's theorem and stability of periodic
solitary waves}, J. Dynam. Differential Equations 21 (2009), no. 3, 555--565.

\bibitem{Pa} J. M. Palacios, \emph{Orbital stability and instability of periodic wave solutions for the $\phi^4$-model}, preprint, arXiv:2005.09523v2

\bibitem{PeSc} M. Peskin, D. Schroeder, \emph{An introduction to quantum field theory}, Addison-Wesley Publishing Company, Advanced Book Program, Reading, MA, 1995




\bibitem{RS} M. Reed, B. Simon, \emph{Methods of modern mathematical physics IV, Analysis of operators}. Academic Press, 1978.

\bibitem{Ri} M. J. Rice, Phys. Lett. A 71,152 (1979).

\bibitem{ShSt} J. Shatah, \emph{Stable standing waves of nonlinear Klein-Gordon equations}, Comm. Math. Phys. 91 (1983), no. 3, 313–327.

\bibitem{ShSt2} J. Shatah, W. Strauss, \emph{Instability of nonlinear bound states}. Comm. Math. Phys. 100 (1985), no. 2, 173–190.

\bibitem{St} J. Sterbenz, \emph{Dispersive decay for the 1D Klein-Gordon equation with variable coefficient nonlinearities}. Trans. Amer. Math. Soc. 368 (2016), no. 3, 2081–2113. 

\bibitem{Va} T. Vachaspati, \emph{Kinks and domain walls}, Cambridge University Press, New York, 2006.
An introduction to classical and quantum solitons.


\end{thebibliography}
\end{document}